\documentclass{amsart}
\usepackage{amssymb}
\usepackage[all]{xy}

\title{Universal flattening of Frobenius}
\author{Takehiko Yasuda}
\address{Department of Mathematics and Computer Science, 
Kagoshima University, 1-21-35 Korimoto, Kagoshima 890-0065, Japan}
\email{yasuda@sci.kagoshima-u.ac.jp}

\theoremstyle{plain}
\newtheorem{thm}{Theorem}[section]

\newtheorem{prop}[thm]{Proposition}
\newtheorem{cor}[thm]{Corollary}
\newtheorem{lem}[thm]{Lemma}
\newtheorem{claim}[thm]{Claim}

\theoremstyle{definition}
\newtheorem{defn}[thm]{Definition}
\newtheorem{ex}[thm]{Example}

\newtheorem{Q}[thm]{Question}

\theoremstyle{remark}
\newtheorem{rem}[thm]{Remark}

\def\AA{\mathbb A}

\def\GG{\mathbb G}

\def\QQ{\mathbb Q}
\def\RR{\mathbb R}

\def\ZZ{\mathbb Z}

\def\ZZpos{\mathbb{Z}_{>0}}
\def\ZZnonneg{\mathbb{Z}_{\ge 0}}

\def\RRpos{\mathbb{R}_{> 0}}
\def\RRnonneg{\mathbb{R}_{\ge 0}}

\def\cO{\mathcal{O}}

\def\cU{\mathcal{U}}
\def\cV{\mathcal{V}}
\def\cW{\mathcal{W}}
\def\cX{\mathcal{X}}

\def\cZ{\mathcal{Z}}

\def\fa{\mathfrak a}
\def\fb{\mathfrak b}
\def\fc{\mathfrak c}
\def\fd{\mathfrak d}

\def\fm{\mathfrak m}

\def\fz{\mathfrak z}

\def\11{\boldsymbol{1}}

\def\sm{\mathrm{sm}}
\def\red{\mathrm{red}}
\def\pr{\mathrm{pr}}
\def\minDelta{\Delta_{\mathrm{min.res.}}}

\def\cocoa{{\hbox{\rm C\kern-.13em o\kern-.07em C\kern-.13em o\kern-.15em A}}}

\DeclareMathOperator{\codim}{codim}

\DeclareMathOperator{\Spec}{Spec}
\DeclareMathOperator{\cSpec}{\mathcal{S}\mathit{pec}}
\DeclareMathOperator{\Hom}{Hom}

\DeclareMathOperator{\initial}{In}
\DeclareMathOperator{\relint}{relint}
\DeclareMathOperator{\FB}{FB}

\DeclareMathOperator{\Hilb}{Hilb}

\DeclareMathOperator{\length}{length}

\def\GHilb{\mathrm{Hilb}^{G}}

\begin{document}

\maketitle

\begin{abstract}
For a variety $X$ of positive characteristic and a non-negative integer $e$,
we define its $e$-th F-blowup to be the universal flattening of 
the $e$-iterated Frobenius of $X$. Thus we have
the sequence (a set labeled by non-negative integers)
of blowups of $X$.
Under some condition, the sequence stabilizes and leads
to a nice (for instance, minimal or crepant) resolution.
For  tame quotient singularities, the sequence leads 
to the $G$-Hilbert scheme.
\end{abstract}

\section{Introduction}

In this paper, we introduce and study a  characteristic-$p$ variant of the higher Nash blowup introduced in \cite{MR2371378}.  Let $X$ be  a $d$-dimensional variety over a perfect field $k$ 
of characteristic $p>0$, $X_\sm$ its smooth locus
 and $F^e : X_e \to X$ the $e$-iterated $k$-linear Frobenius.  
For a point $x \in X_\sm (K)$, the fiber $(F^e)^{-1} (x) \subset (X_{e})_K$ is 
a $0$-dimensional subscheme of length $ q^{d} $ and identified with a $K$-point
of the Hilbert scheme $\Hilb_{q^{d}} (X_e)$ of length $q^{d} $ subschemes of $X_e$.

\begin{defn}
We define the \emph{$e$-th  F-blowup}, $\FB_e (X)$,  to be the closure of the set
\[
\{(F^e)^{-1} (x)  \mid x \in X_\sm \}
\]
 in $\Hilb_{q^{d}} (X_e)$. 
\end{defn}

We will see that there exists a natural morphism $\pi_e:\FB_e (X) \to X$,
which is projective and is an isomorphism exactly over $X_\sm$.
Moreover  $\FB_e (X)$ is isomorphic to
the irreducible component of the relative Hilbert scheme
 $\Hilb _{q^d} (X_e /X)$ that dominates $X$. 
Equivalently it is isomorphic to the universal flattening of  
$F^e : X_e \to X$ or  to  the universal flattening of
 $F^e_* \cO_{X_e}$
(for the universal flattening of a general coherent sheaf, see
\cite{MR1218672,MR0244517,1087.14011}).
The following results can be shown by similar methods as ones in
\cite{MR2371378,math.AG/0610396}:

\begin{prop}\label{prop-intro}
\begin{enumerate}
\item (Propositions \ref{prop-etale}, \ref{prop-completion}, \ref{prop-product}, \ref{prop-smooth} and \ref{prop-field}) For each $e$, the $e$-th F-blowup $\FB_e(X)$ is compatible with \'etale morphisms, completions, products, smooth morphisms and field extensions.
\item (Proposition \ref{prop-separation}) At each point $x \in X$, for $e \gg 0$, 
the $e$-th F-blowup separates
the analytic branches at $x$.
\end{enumerate}
\end{prop}

Given a variety, we obtain a sequence of blowups
\[
X= \FB_{0}(X), \  \FB_{1}(X), \ \FB_{2}(X), \ \dots . 
\]
It seems natural to ask the following questions on the asymptotic behavior 
of this sequence:

\begin{Q}\label{Q-intro}
\begin{enumerate}
\item \label{Q-smooth} For $e \gg 0$, is $\FB_e (X)$ smooth? 
\item  Does the sequence stabilize?
\item \label{Q-simultaneous} Is the sequence bounded? Namely, does there exists a proper birational
morphism $Y \to X$
such that $Y$ dominates all the $\FB_{e} (X)$, $e \ge 0$?
\end{enumerate}
\end{Q}

We can also ask a variant of the first question: 

\begin{Q}\label{Q-intro2}
Does an iteration of F-blowups
\[
\FB_{e_{1}} (X),  \FB_{e_{2}}(\FB_{e_{1}}(X)),  \FB_{e_{3}}(\FB_{e_{2}}(\FB_{e_{1}}(X))),\dots
\]
lead to a smooth variety?
\end{Q}

 We will give partial answers
to these questions for certain classes of singularities.

\begin{rem}
We can ask the same questions for the higher Nash blowups.
My recent computation seems say that the answers are negative:
The sequence of higher Nash blowups of $A_{3}$-singularity
probably does not lead to a smooth variety or
stabilize either. 
\end{rem}

In dimension one, we will obtain the following affirmative result:

\begin{thm}[Theorem \ref{thm-curve-resolution}]
Let $X$ be a one-dimensional variety in positive characteristic. Then for $e \gg 0$, $\FB_{e} (X) $ is smooth. 
\end{thm}

Now we restrict our attention to (not necessarily normal) toric singularities.
From the compatibility with \'etale morphisms, it is enough to consider 
an affine toric variety $ X=\Spec k[A] $
for a finitely generated submonoid $A \subset M:= \ZZ^d$ which generates $M$ as a group.  Here $k$ is an arbitrary field. 
Then we define a Frobenius-like morphism as follows: 
For $l \in \ZZpos$, the inclusion $A \hookrightarrow (1/l) \cdot A$ induces
the morphism 
\[
F^{(l)} : X_{(l)}:=\Spec k[(1/l)\cdot A] \to X.
\]
We define $\FB_{(l)} (X)$ to be the irreducible component of 
$\Hilb_{l^d} ( X_{(l)}/ X) $ which dominates $X$.
If $k$ is a perfect field of characteristic $p>0$ and $l=p^e$, then $F^{(l)}$
 is the same as the $e$-iterated Frobenius and $\FB_{(l)} (X) =\FB_e (X)$. 

\begin{rem}
Fujino \cite{MR2328817} called the morphism $F^{(l)}$ the $l$-th multiplication map and used it to prove vanishing theorems for toric varieties. 
\end{rem}

We will see that $\FB_{(l)} (X)$ is a (non-normal) toric variety. 
The fan associated to $\FB_{(l)}(X)$ is the Gr\"obner fan of some ideal.
Moreover we can give a description of the coordinate rings of affine charts (Proposition \ref{prop-coord-ring}). 
Our method of computation is similar to ones in \cite{MR2356842,ito-yukari-kinosaki}, where the authors
compute $G$-Hilbert schemes.
Using this explicit description of $\FB_{(l)}(X)$,  we will show the following:

\begin{thm}\label{thm-toric-intro}
Let $X$ be a toric variety.
\begin{enumerate}
\item  (Theorem \ref{thm-stability}) If $X$ is normal, then there exists $l_0 \in \ZZpos$ such that for every $l \in \ZZpos$, 
there exists a natural birational morphism $\FB_{(l_0)}(X) \to \FB_{(l)} X$,
 which is an isomorphism if $l \ge l_0$.
 \item (Theorem \ref{thm-simultaneous}) 
 The sequence of blowups, $\FB_{(1)}(X),\, \FB_{(2)}(X),\dots$,  is bounded. Namely
 there exists a toric proper birational morphism $Y \to X$ which factors
 as $Y \to \FB_{(l)} (X) \to X$ for every $l \in \ZZpos$.
\item \label{1} (Theorem \ref{thm-minimal}) If $X$ is normal and two-dimensional, then for $l \gg 0$,  $\FB_{(l)} (X)$ is the minimal resolution of $X$.
(For a slight generalization, see Remark \ref{rem-minimal}.)
\item (Proposition \ref{prop-isolated}) If $X$ has only an isolated singularity and the normalization $\tilde X$ of $X$ is an affine space, then 
for $l \gg 0$, $\FB_{(l)} (X)$ is the blowup of $\tilde X$ at the origin. 
\end{enumerate}
\end{thm}

The F-blowup is also closely related to the $G$-Hilbert scheme introduced
by Ito and Nakamura \cite{MR1420598}.
Suppose that a finite group $G$ with $p \nmid |G|$
effectively acts on a smooth variety $M$. Let $X:=M/G$ be the quotient
variety. The  $G$-Hilbert scheme, denoted $\GHilb (M)$,
is the closure of the set of free orbits in the Hilbert scheme of $M$.

\begin{thm}[Theorem \ref{thm-GtoF-isom}]\label{thm-intro-G}
For each $e$, there exists a projective birational morphism $\GHilb (M) \to \FB_{e}(X)$.
For sufficiently large $e$, the morphism is an isomorphism.  
\end{thm}

Thus for every quotient singularity by a tame finite group action,
the answer to the second question of Question \ref{Q-intro} is affirmative.
We have thus obtained a new interpretation of the $G$-Hilbert scheme
as the universal flattening of Frobenius.
We can easily generalize Theorem \ref{thm-intro-G} to Deligne-Mumford
stacks (Theorem \ref{cor-stack}).

Thanks to  Theorem \ref{thm-intro-G}, we deduce nontrivial facts
on the $G$-Hilbert scheme from properties of the F-blowup, and vice versa:
\begin{enumerate}
\item   The $G$-Hilbert scheme depends only on the quotient variety.
\item Since the $G$-Hilbert scheme is not generally normal \cite{MR2356842},
the answer to the desingularization problem is negative.
\item
 If $M=\AA^{2}_{k}$  and 
$G \subset GL(2,k)$,
then the $G$-Hilbert scheme is the minimal resolution of $X$.
This is a slight generalization of results in \cite{MR1916656,MR1714824,MR1815001}
to groups possibly containing reflections.
\item If $M=\AA^{3}_{k}$  and $G \subset SL(3,k)$, then for $e \gg 0$,
$\FB_{e}(X)$ is a crepant resolution, which follows from \cite{MR1824990,MR1838978}.
\end{enumerate}

To simplify the notation and make it consistent with subsequent works  \cite{Yasuda:arXiv:0810.1804,yasuda-monotonicity},
we consider the Frobenius morphism corresponding to 
the inclusion map $\cO_{X} \hookrightarrow \cO_{X}^{1/p^{e}}$
of sheaves in this version of the manuscript, while
we considered $\cO_{X}^{p^{e}} \hookrightarrow \cO_{X}$
in the previous versions. Those who have read the previous versions
have to be careful about the notational changes caused by this. 

New proofs of some results in this article have been found in \cite{Yasuda:arXiv:0810.1804,yasuda-monotonicity}. Compared to them, the proofs in this article
are more primitive.

The article is organized as follows. 
In Section \ref{sec-basic} we establish basic properties of F-blowup and prove the desingularization of a curve. 
In Section \ref{sec-toric} we treat the toric case and prove main results,
which are partial answers to the questions raised above.
In Section \ref{sec-G} we discuss a relation between the $G$-Hilbert
scheme and the F-blowup.

\subsection*{Convention}
A \emph{variety} means a separated integral scheme of finite type
over a field. A \emph{cluster} means a zero-dimensional subscheme.
We always denote by $p$ the characteristic of the base field.
Given a non-negative integer $e$, we denote the $e$-th power $p^{e}$ of $p$ by $q$.
For a scheme $X$ over a perfect field of characteristic $p > 0$,
we write $X_{e} = \cSpec \cO_{X}^{1/q}$. 
The \emph{$e$-iterated $k$-linear Frobenius},
denoted  $F^{e} =F^{e}_{X}:X_{e} \to X$, is the morphism corresponding to the inclusion $\cO_{X} \hookrightarrow \cO_{X}^{1/q}$. 
We often call this simply the \emph{$e$-th Frobenius}.
For a  closed subscheme $Z \subset X$ with defining ideal (sheaf) $I$,
denote by $Z^{[q]}$ the closed subscheme of $X$ defined by
the Frobenius power $I^{[q]}$, which is generated by the $q$-th power of the sections of $I$. 
We often omit  the subscript $k$ of the fiber product $\times_k$. 

\subsection*{Acknowledgments}

This work was mainly done in 2007 when I was a JSPS research fellow (PD) at RIMS, Kyoto University.
 I wish to thank Shigeru Mukai for stimulating conversations and  helpful suggestions.
 I am also indebted to Akira Ishii and Yukari Ito
 for helpful comments concerning the $G$-Hilbert scheme,
  and Ken-ichi Yoshida for ones concerning the Hilbert-Kunz function.
 During this work, I made many computations by CoCoA \cite{CocoaSystem}  and the function ``Toric'' \cite{MR1681344} implemented in it, which were very  suggestive.

\section{Basic properties}\label{sec-basic}

Arguments in this section are similar to ones for the higher Nash blowup
in \cite{MR2371378,math.AG/0610396}.
We work over a perfect field $k$ of characteristic $p > 0$.

\subsection{Construction}

Let $X$ be a variety of dimension $d$ and 
\[
\Gamma \subset (X_\sm)_e \times X_\sm \subset X_e\times X_\sm
\]
 the graph of the $e$-th Frobenius of $X_{\sm}$.  
Since $\Gamma$ is a flat family of clusters in $X_{e}$  over $X_\sm$, 
there exists the corresponding morphism of $k$-schemes,
\[
 \iota : X_\sm \to \Hilb_{q^d} (X_{e}),
\]
which maps a point $x \in X_{\sm}(K)$ to the scheme-theoretic fiber $(F^{e})^{-1}(x) \subset (X_{e})_{K}$.
Note that $(F^{e})^{-1}(x) = ((F^{e})^{-1}(x) _{\red})^{[q]} $.

\begin{lem}
$\iota$ is an immersion.
\end{lem}

\begin{proof}
Without loss of generality, 
we may suppose that $k$ is algebraically closed.  
Since $\iota$ is clearly injective, it suffices to show that 
for every $x \in X_\sm (k)$, the map of the tangent spaces
\[
 T\iota : T_x X \to T _{\iota (x)} \Hilb _{q^d} (X_{e})
\]
is injective. 
If $x_1,\dots, x_d \in \cO_{X,x}$ are local coordinates of $X$ at $x$,
then   $\Gamma$ is locally defined by the ideal
\[
 \langle x_i \otimes 1- 1\otimes x_i \mid 1 \le i \le d \rangle \subset 
\cO_{X,x}^{1/q} \otimes \cO_{X,x} .
\]
The tangent space $T_x X$ is generated by the tangent vectors
\[
\epsilon _i :  \Spec \cO_{X,x} /\langle  x_i^{2},x_j \mid j \ne i \rangle \hookrightarrow X, \, i=1,\dots,d.
\]
The base change of $\Gamma$ by $\epsilon _i$ is defined by
\[
 \langle x_i \otimes 1- 1\otimes x_i,  1\otimes x_j \mid j \ne i \rangle \subset 
\cO_{X,x}^{1/q} \otimes \cO_{X,x} /\langle  x_i^{2},x_j \mid j \ne i \rangle .
\]
Therefore if we identify $T_{\iota(x)}  \Hilb _{q^d} (X_{e})  $ with 
\[
\Hom ( \langle x_1,\dots, x_d \rangle , \cO^{1/q}_{X,x} / \langle x_1,\dots, x_d\rangle  ),
\]
then $T\iota (\epsilon_i)$ maps $x_i \mapsto -1$ and $x_j \mapsto 0$, $j \ne i$ (see for instance \cite[Proof of Theorem VI-29]{MR1730819}). 
It follows that the $T\iota(\epsilon
_{i})$, $i=1,\dots,d$, are linearly independent, and so $T\iota$ is injective. We have completed the proof.
\end{proof}

\begin{defn}
We define the \emph{$e$-th F-blowup} of $X$, denoted $\FB_e ( X)$, to be the closure of $\iota (X_\sm)$.
\end{defn}

\begin{prop}
The birational map $\FB_e (X) \dasharrow X$ is extended to 
a morphism $\FB_e (X) \to X$.
\end{prop}

\begin{proof}
Let $R:=\FB _{e} (X)$ and $Y \subset R \times X$ be the graph closure of the birational map.
We need to show that the projection $\phi:Y \to R$ is an isomorphism.
Again we may suppose that $k$ is algebraically closed. 
It is easy to see that $\phi$ is bijective.
So it suffices to show that 
for every closed point $y \in Y$,
the map of tangent spaces
\[
 T\phi :T_{y} Y \to T_{\phi(y)} R
\]
is injective. Let $x \in X$ and $r \in R$ be the images of $y$.
Then $\cO_{Y,y}$ is the subalgebra of the common quotient field of $R$ and $X_{e}$ that is generated
by $\cO_{X,x}$ and $\cO_{R,r}$. To obtain a contradiction,
suppose that there exists a nonzero tangent vector at $y$,
\[
\epsilon : \Spec k[t]/\langle t^{2}\rangle \hookrightarrow Y,
\]
which maps to the zero tangent vector at $r$.
Let $\epsilon^{*}:\cO_{Y,y} \twoheadrightarrow k[t]/\langle t^{2} \rangle$ be
the map corresponding to $\epsilon$. 
Since $\epsilon$ maps to a nonzero tangent vector at $x$, 
 there exists $f \in \cO_{X,x}$ 
with $\epsilon^{*}(f)=t$. Let $\cU \subset  X_{e} \times Y$ be the family of clusters
in $X_{e}$ corresponding to $\phi$. Over $X_{\sm} \subset Y $, $\cU$ coincides with
the graph of $ (X_{\sm})_{e} \to X_{\sm}$. So $\cU$ is 
the closure of this graph. Therefore the defining ideal of $\cU$ has the local section
\[
 f \otimes 1 - 1 \otimes f \in\cO^{1/q}_{X,x} \otimes \cO_{Y,y}  .
\]
Then, since $\epsilon^{*}(f)=t$, the pull-back  $\cU_{\epsilon}$ of $\cU$ by $\epsilon$
is defined by an ideal of 
\[
\cO^{1/q}_{X,x}\otimes  k[t]/\langle t^{2} \rangle  =\cO^{1/q}_{X,x}[t]/\langle t^{2}\rangle
 \]
 which contains $f - t$.
Let $Z \subset X_{e}$ be the cluster corresponding to $r$.
Then its defining ideal $I_{Z} \subset \cO_{X,x}^{1/q}$ contains $f$.
If we think of $T_{r} R$ as a subspace of $\Hom (I_{Z},\cO^{1/q}_{X,x}/I_{Z})$,
then $T\phi(\epsilon)$ maps $f $ to $-1 $ (see for instance  \cite[Proof of Theorem VI-29]{MR1730819} again). 
Thus $T\phi(\epsilon)$ is nonzero, a contradiction.  
\end{proof}

 \begin{cor}\label{cor-F-nbhd}
 For every point $Z \in (\FB_e (X))(K)$ with $K/k$ a field extension, 
the cluster $Z \subset X_K := X\otimes _k K$ is set-theoretically one $K$-point and scheme-theoretically contained in the cluster
$(Z_\red )^{[q]}$. 
 \end{cor}
 
 \begin{proof}
 Let $\Gamma \subset  X_{e}\times X $ be the graph of the $e$-th Frobenius.
 Then the fiber of the projection $\Gamma  \to X$ over $x$
 is identical to $(F^{e})^{-1}(x) = ((F^{e})^{-1}(x)_{\red})^{[q]}$.
 Let $\Gamma' \subset   X_{e} \times \FB_{e}(X)$ be the pull-back of $\Gamma$
 and let $\cU \subset  X_{e} \times \FB_{e}(X) $ be the universal family of clusters in $X_e$ over $\FB_{e}(X)$.  
Then we have $\cU = (\Gamma')_{\red}  \subset \Gamma'$, which proves the corollary. 
 \end{proof}
 
\begin{prop}\label{prop-relative}
The $X$-scheme $\FB_e (X)$ is isomorphic to
the irreducible component of the relative Hilbert scheme 
$\Hilb_{q^d}(X_{e}/X )$
that dominates $X$.
\end{prop}

\begin{proof}
Consider  the graph of $ \FB_e (X) \to X$,
\[
G \subset \FB_e (X) \times X \subset \Hilb_{q^d} (X_{e}) \times X = \Hilb_{q^d} (X_{e} \times X /X),
\]
which is, by definition, isomorphic to $\FB_e (X)$ as an $X$-scheme. 
If $\Gamma \subset X_{e} \times X$ is the graph of the $e$-th Frobenius $X_{e} \to X$,
then we have
\[
\Hilb_{q^{d}}(X_{e}/X)\cong \Hilb_{q^d} (\Gamma / X) \hookrightarrow \Hilb_{q^d} (X_{e} \times X/X) .
\]
Now $G$ and the irreducible component of $\Hilb_{q^d} (X_{e}/X)$ 
determines the same subscheme of $\Hilb_{q^d} (X_{e} \times X /X)$, which proves the proposition.
\end{proof}

\begin{cor}\label{cor-inclusion}
 The morphism $\FB_e (X) \to X$ is projective and birational, moreover an isomorphism over $X_{\sm}$.
\end{cor}

\begin{proof}
This follows from Proposition \ref{prop-relative}.
\end{proof}

The relative Hilbert scheme $\Hilb_{q^d}(X_{e}/X)$
is canonically isomorphic to the Quot scheme $\mathrm{Quot}_{q^d} ( \cO_X ^{1/q})$ of the coherent $\cO_{X}$-module $\cO_X^{1/q}$. 
Hence $\FB_e (X)$ is  the universal flattening of $\cO_{X}^{1/q}$.
For the universal flattening of a general coherent sheaf, see \cite{MR1218672,MR0244517,1087.14011}.

\subsection{Kunz's criterion for smoothness}

By construction, if $X$ is smooth, for every $e >0$, the morphism $\FB_e (X) \to X$ is an isomorphism. In fact, the converse is also true. 
Kunz's criterion \cite{MR0252389} says that
the following are equivalent:
\begin{enumerate}
\item $X$ is smooth.
\item For some $e \in \ZZpos$, $F^e_{X} $ is flat.
\item For every $e \in \ZZpos$, $F^e_{X} $ is flat.
\end{enumerate}
Since $\FB_e (X) \to X$ is a flattening of $F^e_{X}:X \to X$,
if $F^e_{X}$ is not flat, then $\FB_e (X) \to X$ cannot be an isomorphism. 
Thus we have:

\begin{prop}\label{prop-Kunz}
The following are equivalent:
\begin{enumerate}
\item $X$ is smooth.
\item For some $e \in \ZZpos$, $\FB_e (X) \to X$ is an isomorphism.
\item For every $e \in \ZZpos$, $\FB_e (X) \to X$ is an isomorphism.
\end{enumerate}
\end{prop}

This gives a partial answer to Question \ref{Q-intro2}:

\begin{cor}
Let $X$ be a one-dimensional variety and
$e_{i}$, $i \in \ZZpos$, an infinite sequence of positive integers.
Then the sequence, $\FB_{e_{1}} (X)$, $ \FB_{e_{2}}(\FB_{e_{1}}(X))$, $ \FB_{e_{3}}(\FB_{e_{2}}(\FB_{e_{1}}(X))),\dots$,
leads to a smooth variety.
\end{cor}

\subsection{Compatibility}

\subsubsection{\'Etale morphism}

\begin{prop}\label{prop-etale}
If $Y \to X$ is an \'etale morphism of varieties,
then there exists a natural isomorphism
$\FB_e (Y) \cong \FB_e (X) \times_{X} Y$. 
\end{prop}

\begin{proof}
We have $Y_{e} \cong  Y \times _{X} X_{e} $.
Hence   $\Hilb_{q^d} (Y_{e}/Y) \cong \Hilb _{q^d}(X_{e} /X ) \times_{X} Y$
and the proposition follows.
\end{proof}

\subsubsection{Completion}

Let $X$ be a variety, $K/k$ a field extension, $x \in X(K)$ a $K$-point
and $\hat X := \Spec \hat \cO_{X_K,x}$ the completion of $X_K= X \otimes_k K$ at $x$.

\begin{defn}
We define $\FB_e (\hat X)$ to be the union of those irreducible components
of $\Hilb _{q^d} (\hat X _{e} /\hat X)$ that dominate an irreducible component
of $\hat X$. For an irreducible component $W \subset \hat X$,
we define $\FB_e (W)$ to be the irreducible component of
$\Hilb _{q^d} (\hat X _{e} /X)$ that dominates $W$.
\end{defn}

By definition, if $W_i$, $1 \le i \le l$, are the irreducible components of $\hat X$,
then $\FB_e (\hat X) = \bigcup _{i=1}^l \FB_e (W_i)$. 
It is not generally a disjoint union as in the following example.

\begin{ex}
Suppose that  $\hat X$ is of dimension one and has two regular irreducible components $W_1,\, W_2 (\cong \Spec k[[t]])$ intersecting non-transversally.  
If $p=2$, then $\FB_1(W_1)$ and $\FB_1(W_2)$ share the closed point,
which corresponds to the common subscheme of $W_1$ and $W_2$,
\[
W_1 \supset \Spec k[t]/\langle t^2 \rangle \subset W_2.
\]
\end{ex}

\begin{prop}\label{prop-completion}
We have a natural isomorphism 
\[
 \FB_e (\hat X) \cong \FB_e (X) \times _{X} \hat X.
\]
\end{prop}

\begin{proof}
The method of proof is the same as in the proof of Proposition \ref{prop-etale}.
\end{proof}

\subsubsection{Product}

\begin{prop}\label{prop-product}
For varieties $X$ and $Y$, 
there exists a natural isomorphism $\FB_e (X \times Y) \cong \FB_e (X) \times \FB_e (Y)$.
\end{prop}

\begin{proof}
We first prove that there exists a morphism
\[
\psi: \FB_e (X) \times \FB_e (Y) \to \FB_e (X \times Y),\ (V,W) \mapsto V \times W.
\]
Let $ \cV \subset  X_{e} \times \FB_{e}(X) $ and $\cW \subset   Y_{e} \times \FB_{e}(Y)$
be the universal families over $ \FB_{e}(X)$ and $ \FB_{e}(Y)$ respectively. 
Then we consider their product 
\[
\cV \times \cW \subset  X_{e} \times Y_{e} \times \FB_{e}(X) \times \FB_{e}(Y)  ,
\]
which we regard as a family of clusters in $X_{e} \times Y_{e}$ over $\FB_{e}(X) \times \FB_{e}(Y)$.
This family is flat because every fiber has length $q^{\dim X + \dim Y}$.
Moreover the restriction of this over  $\FB_{e}(X_{\sm}) \times \FB_{e}(Y_{\sm})$ is 
identical to the graph of the $e$-th Frobenius of $X\times Y$,
\[
X_{e} \times Y_{e} =(X \times Y)_{e} \to X \times Y.
\]
Hence we have the desired morphism $\psi$.

To show that $\psi$ is an isomorphism, we may and shall assume that $k$ is algebraically closed again.
Being proper and birational, $\psi$ is surjective.
Thus every closed point $Z \in \FB _e (X \times Y)$ is of the from $V \times W$.
If $\pr_1 :X \times Y \to X$
and $\pr _2 : X\times Y \to Y$ denote the projections,
then $\pr_1 (Z) =V$ and $\pr_2 (Z) =W$.
Hence $\psi$ is bijective. 

Let 
\[
\epsilon: \Spec k[t]/\langle t^{ 2 }\rangle \to  \FB_{e}(X) \times \FB_{e} (Y)
\] 
be a nonzero tangent vector
and $\epsilon _{X}$ and $\epsilon_{Y}$ its images on $\FB_{e}(X)$ and $\FB_{e}(Y)$ respectively.
Then $ \epsilon _{X} $ or  $\epsilon_{Y}$, say $\epsilon_{X}$ is nonzero. 
Then we have an equality of the induced first order families,
\[
(\cV \times \cW)_{\epsilon} = \cV_{\epsilon_{X}} \times_{\Spec k[t]/\langle t^{2} \rangle } \cW_{\epsilon_{Y}}.
\]
Hence  $\cV_{\epsilon_{X}}$, which is a nontrivial family, is
the image of $(\cV \times \cW)_{\epsilon}$
by the projection. This shows that $(\cV \times \cW)_{\epsilon}  $
is also a nontrivial family and $T\psi(\epsilon)$ is nonzero.
Thus the tangent maps of $\psi$ are injective, which completes the proof. 
\end{proof}

\subsubsection{Smooth morphism}

\begin{prop}\label{prop-smooth}
Let $Y \to X$ be a smooth morphism of varieties.
Then we have a natural isomorphism $\FB_e (Y) \cong \FB_e (X) \times _{X} Y$.
\end{prop}

\begin{proof}
The smooth morphism  is, \'etale locally on $Y$, isomorphic to 
the projection $X \times \AA ^ c \to X$.
Hence the proposition follows from the compatibility of $\FB_e(-)$ with
\'etale morphisms and products.
\end{proof}

\begin{rem}
The compatibility with products and smooth morphisms
 holds not for the higher Nash blowup, but for the flag higher Nash blowup
 introduced in \cite{math.AG/0610396}.
\end{rem}

\subsubsection{Field extension}

\begin{prop}\label{prop-field}
Let $K/k$ be a field extension with $K$ also perfect. Then $\FB_e (X_K) \cong (\FB_e (X))_K$.
\end{prop}

\begin{proof}
Since $F^e_{X_{K}}$ is the base change of $F^e_{X}$
by the $\Spec K \to \Spec k$, we have 
\[
 \Hilb _{q^d } ((X_{K})_{e}/X_{K}) \cong \Hilb_{q^d} (X_{e}/X) \otimes _k K,
\]
and the proposition follows.
\end{proof}

\subsection{Separation of analytic branches}

Suppose that $k$ is algebraically closed.
Let $X$ be a variety, $\hat X$ the completion of $X$ at $x \in X(k)$
and $W_i \subset \hat X$, $ 1 \le i \le l$, the irreducible components.
Since $\FB_e ( \hat X) = \bigcup _i \FB_e (W_i)$, 
for every point $Z \in (\FB_e (\hat X))(k)$, there exists $i$ with
$Z \subset (W_i)_{e}$.
Identifying $W_{i}$ and $(W_{i})_{e}$, we have $Z \subset W_i$.
If $Z \not \subset W_i \cap W_j$ for every $i \ne j$, where $W_i \cap W_j$ is the 
\emph{scheme-theoretic} intersection,
then the analytic branch containing $Z$ is unique.
If for every $Z \in (\FB_e (\hat X))(k)$ and for every $i \ne j$, 
$Z \not \subset W_i \cap W_j$, then $\FB_e (\hat X )= \bigsqcup  _i \FB_e( W_i)$, the disjoint union.
If it is the case, we say that \emph{the $e$-th F-blowup separates the analytic branches at $x$}.

\begin{prop}\label{prop-separation}
For $e \gg 0$, the $e$-th F-blowup separates the analytic branches at $x$.
\end{prop}

\begin{proof}
Let $A:= \bigcup _{i \ne j} W_i \cap W_j$, which is a closed subscheme of $\hat X$
of dimension $<d$. It suffices to show that for $e \gg 0$,
every $Z \in (\FB_e (\hat X))(k)$ satisfies $Z \not \subset A$.
From Corollary \ref{cor-F-nbhd}, for every $Z \in (\FB_e (X))(k)$,
we have $Z \subset x ^{[q]}$.
It follows that if $Z \subset A$, then $Z \subset A \cap x^{[q]}$ 
and 
\[
p^{de} = \length \cO_Z \le \length \cO_A /\fm_A^{[p^e]}.
\]
Here $\cO_Z$ and $\cO_A$ are the coordinate rings of $Z$ and $A$ respectively,
and $\fm_A \subset \cO_A$ is the maximal ideal.
The function $f(e)=  \length \cO_A /\fm_A^{[p^e]}$
is called the Hilbert-Kunz function.
Monsky's theorem \cite{MR697329} says 
\[
f(e) =O (p^{e \cdot \dim A}).
\]
So, for $e \gg 0$, the above inequality does not hold
and for every  $Z \in (\FB_e (\hat X))(k)$, we have $Z \not \subset A$.
This completes the proof.
\end{proof}

It seems an open problem whether for a scheme $T$, the Hilbert-Kunz functions
of the local rings $\cO_{T,t}$, $t \in T(k)$ is
bounded from above. 
If the answer is affirmative, then for $e \gg 0$, the $e$-th F-blowup separates
the analytic branches simultaneously at all points of $X$. 
This can be shown by applying arguments in \cite{MR2371378}.
The boundedness of the Hilbert-Kunz \emph{multiplicity} 
is proved by Enescu and Shimomoto \cite{MR2119113} under some condition.

\subsection{The resolution of singularities of a curve}

\begin{thm}\label{thm-curve-resolution}
Let $X$ be a one-dimensional variety. Then for $e \gg 0$, $\FB_e(X)$ is smooth.
\end{thm}

\begin{proof}
Let $\hat X$ be the completion of $X$ at some point as above.
From Proposition \ref{prop-separation}, it is enough to consider the case where
$\hat X$ is irreducible. Then $\hat X = \Spec R$ and its normalization is
$Y := \Spec k[[t]]$.  
The  morphism $Y \to X$ factors as 
\[
 Y \xrightarrow{h_{e}} \FB_e (\hat X) \to \hat X.
\]
Suppose now that $e$ is so large that $t^{p^e} \in R$, equivalently 
$t\in R^{1/p^e}$. 
Then we claim that the subscheme $\cZ\subset Y \times \hat X_e$ defined by the ideal
\[
 \langle  t \otimes 1 - 1 \otimes t  \rangle \subset R^{1/p^e} \otimes k[[t]]
\]
 is the family of clusters in $X_{e}$ over $Y$ corresponding to $h_{e}$. 
Indeed the restriction of $\cZ$ to the generic point is obviously
  the pullback of  the universal family over $\FB_{1}(X)$ to $\Spec k((t))$. Moreover the closed fiber of  $\cZ \to Y$
  is $\Spec R^{1/p^e}/\langle t \rangle$ and has length $p^{e}$ from \cite[Lem.\ 11.12]{MR1322960}.
Therefore $\cZ$ is flat over $Y$ and  hence the claim holds.

Let $\cZ_{\epsilon} \subset \Spec R^{1/p^{e}} \otimes k[t]/\langle t^{2}\rangle $ be 
the restriction of $\cZ$ to the canonical tangent vector 
\[
\epsilon : \Spec k[t]/\langle t^{2} \rangle \hookrightarrow Y = \Spec k[[t]].
\]
Then the induced tangent vector of $\FB_{e}(X)$, 
\[
Th_{e}(\epsilon)= h_{e} \circ \epsilon:\Spec k[t]/\langle t^{2} \rangle \hookrightarrow Y \to \FB_{e}(X),
\]
corresponds to a homomorphism
\[
 \langle t \rangle_{R^{1/p^{e}}} \to R^{1/p^{e}} / \langle t \rangle, 
\]
which maps $t$ to $-1$. In particular, $Th_{e}(\epsilon) \ne 0$ and hence $h_{e}$ is an isomorphism.
We have proved the theorem.
\end{proof}

\section{F-blowups of toric varieties}\label{sec-toric}

In this section we suppose that $k$ is an arbitrary field, say of
characteristic $p \ge 0$.

\subsection{Toric and Gr\"obner computation of the F-blowup}

\subsubsection{Frobenius-like morphisms of toric varieites}

Let $M = \ZZ^d$ be a free abelian group of rank $d$ and $A \subset M$  a finitely generated submonoid
which generates $M$ as an abelian group. 
The group algebra $k[M]$ is identified with the Laurent polynomial ring $k[x_1^\pm,\dots,x_d^\pm]$ and $k[A] \subset k[M]$ is
a subalgebra. Put $X := \Spec k[A]$,
which is a (not necessarily normal) affine toric variety
and contains the torus $T:=\Spec k[M]$ as an open subvariety.
For $l \in \ZZpos$, the inclusions $M \hookrightarrow (1/l)\cdot  M$ and $A \hookrightarrow  (1/l)\cdot  A$  
induce the morphisms of $k$-schemes,
\begin{gather*}
F^{(l)}_{T} :T_{(l)} :=\Spec k[(1/l)\cdot M] \to T, \text{ and} \\
F^{(l)}_{X} :X_{(l)}:=\Spec k[(1/l)\cdot A] \to X .
\end{gather*}
If $k$ is perfect and if $p >0$ and $l=p^e$, then $F^{(l)}_{X} =F^e_{X}$. 
In what follows, we simply write $F^{(l)}$
for both $F^{(l)}_{T}$ and $F^{(l)}_{X}$.

The point $t:=(t_{1},\dots,t_{d}) \in T_{K}(K)$
is defined by the ideal
\[
 \langle  x_{1}-t_{1} , \dots, x_{d}-t_{d} \rangle_{K[M]},
\]
and $(F^{(l)})^{-1}(t)$  by
\[
 \langle x_{1}-t_{1} ,\dots,x_{d}-t_{d} \rangle_{K[(1/l)\cdot M]} .
\]

\begin{defn}
We define $\FB_{(l)}( X)$ to be the closure of
$\{ (F^{(l)})^{-1}(t) \mid t \in T \}$ in the Hilbert scheme $\Hilb_{l^{d}}(X)$.
\end{defn}

Clearly if $k$ is perfect, $p >0$ and $l=p^e$, then $\FB_{(l)} (X) =\FB_{e} (X)$.

\begin{prop}
\begin{enumerate}
\item There exists a natural projective birational morphism $\FB_{(l)} (X) \to X$.
\item $\FB_{(l)}( X)$ is isomorphic to the irreducible component of $\Hilb_{l^{d}} (X_{(l)}/X)$
 dominating $X$.
\item $\FB_{(l)}(-)$ is compatible with open immersions and products.
\end{enumerate}
\end{prop}

\begin{proof}
These results can be proved in the same way as the corresponding ones for
$\FB_{e} (X)$.
\end{proof}

In particular, from the third assertion of the proposition,
 we need only consider the case where $A$ contains no nontrivial group.
(Conversely suppose that  $A$ contains a nontrivial group.
Let $B \subset A$ be the maximal subgroup
and let  $a_{1} , \dots, a_{m},b_{1}, \dots,b_{n } $ be 
generators of $A$ such that $a_{i} \notin B$ and $b_{i} \in B$.
Then there exists a subset of $\{b_{1}, \dots, b_{n}\}$,
say $\{b_{1}, \dots, b_{l}\}$, $l \le n$,
which generates a monoid containing no nontrivial group
but still generates $B$ as a group.
Let  $A'$ be the monoid generated by $a_{1} , \dots, a_{m},b_{1}, \dots,b_{l } $,
which contains no nontrivial group. 
Then $k[A]$ is a localization of $k[A']$ by an element. 
Indeed if we put $b:= \sum_{i=1}^{l} b_{i}$,
then  $k[A ] = k[A'][x^{-b}]$. Thus the toric variety associated to $A$
is an open subvariety of the one associated to $A'$.)

The universal family over $T \subset \FB_{(l)} (X)$ is the subscheme
of $T_{(l)} \times T$ defined by the ideal
\[
 \langle x^{a} \otimes 1  - 1 \otimes x^{a}  \mid a \in M  \rangle \subset k[(1/l) \cdot M] \otimes k[M].
\]
Let $\11 :=(1,\dots,1)\in T$ be the unit point and $(F^{(l)})^{-1}( \11) $
its scheme-theoretic inverse image by $F^{(l)}$.
The defining ideal of $(F^{(l)})^{-1}( \11)$ as a subscheme of $T_{(l)}$ is
\[
\langle x_1-1 ,\dots, x_d-1\rangle _{k[(1/l)\cdot M]}.
\]
Pulling it back to $X_{(l)}$, we obtain the defining ideal of it as a subscheme of
$X_{(l)}$,
\[
\fa_{l} := k[(1/l)\cdot A] \cap \langle x_1-1 ,\dots, x_d-1\rangle _{k[(1/l)\cdot M]} .
\]

\begin{lem}\label{lem-q-binomial}
Let $a,b \in (1/l) \cdot A$ be such that $a-b \in M$.
Then $x^a - x^b \in \fa_{l}$.
\end{lem}

\begin{proof}
We have  
\[
 x^{a-b} -1  \in  \langle x_1 -1 , \dots, x_d -1 \rangle _{k[M]}
 \subset \langle x_1 -1 , \dots, x_d -1 \rangle _{k[(1/l)\cdot M]}.
\]
Therefore 
\[
x^{a} -x^{b}=x^b (x^{a-b}-1) \in \langle x_1 -1 , \dots, x_d -1 \rangle _{k[(1/l)\cdot M]}.
\]
On the other hand, we obviously have $x ^a -x^b \in k[(1/l)\cdot A]$, which 
proves the lemma.
\end{proof}

We will see below (Proposition \ref{prop-grobner}) that binomials as in the lemma
generate $\fa_{l}$. In particular, $\fa_{l}$ is a binomial ideal in $k[A]$.

\subsubsection{Distinguished clusters, Gr\"obner bases and Gr\"obner fans}

For the Gr\"obner basis theory, we refer the reader to \cite{MR1363949}.
In particular, for Gr\"obner bases in
 the monoid algebra $k[A]$, see Chapter 11, \textit{ibid.} 

Let $N:= M^\vee = \Hom_{\mathrm{ab.gp.}} (M,\ZZ)$ be the dual abelian group of $M$
and put $M_\RR := M \otimes \RR$ and $N_\RR := N \otimes \RR$.
Let $A_\RR \subset M_\RR$ be the cone spanned by $A$
and $A_\RR^\vee \subset N_\RR$ the dual cone.
For a Laurent polynomial $f \in k[M]$ and a weight vector $w \in N_{\RR}$,
we define the initial form, $\initial_{w} f$, as usual.
For an ideal $\fc \subset k[A]$, the initial ideal, $\initial_{w} \fc \subset A$,
is the ideal generated by $\initial_{w} f$, $f \in \fc$. 
For a total monomial order $>$ on $k[A]$, we similarly define
the initial term, $\initial_{>}f$, and the initial ideal, $\initial_{>} \fc$.

\begin{defn}
Let $\fc \subset k[(1/l) \cdot A]$ be an ideal, $w \in N_{\RR}$ a weight vector
and $>$ a total monomial order on $k[A]$.
A  subset $I \subset \fc$ is called a \emph{Gr\"obner basis} of $\fc$ with respect to
$>$ (resp.\ $w$)
if $\initial_{>} \fc = \langle \initial _{>} f \mid f\in I \rangle $ (resp.\ $\initial_{w} \fc = \langle \initial _{w} f \mid f \in I \rangle $).
Unlike the standard terminology, we do not assume that $I$ is finite.
\end{defn}

A Gr\"obner basis $I$ of $\fc$ with respect to some weight vector $w$
is also a usual basis of $\fc$.

Having the torus action derived from the one on $X$,  
$\FB_{(l)} (X)$ is also a toric variety and so determines
 a fan $\Delta_{A,l} $ with $|\Delta _{A,l}| = A_\RR^\vee$. 
(Note that the fan does not determine $X$, because $X$ may not be normal.)
For a cone  $\sigma \in \Delta_{A,l}$, we denote by $U_\sigma \subset \FB_{(l)} X$ the corresponding affine open toric subvariety. 
We also denote by $\relint \sigma$ the relative interior of $\sigma$. 
For $w \in \relint \sigma \cap N$, let
\[
\lambda_{w} : \GG_{m} \to T
\]
be the corresponding one-parameter subgroup (see \cite[\S 2.3]{MR1234037}).
Then the distinguished point of $ U_\sigma$ corresponds to the limit cluster
\[
Z_{\sigma} := \lim _{ t \to 0} (F^{(l)})^{-1} ( \lambda_{w}(t)).
\]
Hence $Z_\sigma \subset X$ has the defining ideal $\fb_\sigma:=\initial _w \fa_{l}$. For a nice explanation of this fact, we refer
the reader to \cite[\S 15.8]{MR1322960}.
 Our situation is a little more general than the one
in \textit{ibid}. However the same arguments can apply.
As a consequence, we obtain:

\begin{prop}
The fan $\Delta_{A,l} $ is the  Gr\"obner fan of $\fa_{l}$. 
\end{prop}

\begin{prop}\label{prop-grobner}
For every $w \in \relint A_\RR^\vee $, 
there exists a finite subset of
\[
\Lambda := \{ x^a -x^b | a,b \in (1/l) \cdot A, \, a-b \in M \}
\]
which is a Gr\"obner basis of $\fa_{l}$ with respect to $w$. Hence  $\Lambda$ is also a Gr\"obner basis
of $\fa_{l}$ with respect to $w$.
\end{prop}

\begin{proof}
We just apply the extrinsic computation of Gr\"obner bases
\cite[Algorithm 11.24]{MR1363949}, which goes as follows:
Let $ \phi:k[y_1^{1/l}, \dots,y_n^{1/l}] \to k[(1/l) \cdot A] $ be a surjective $k$-algebra homomorphism
which sends $y_i^{1/l}$ to monomial generators of $k[(1/l)\cdot A]$. 
We give an $\RRnonneg$-grading to $k[(1/l) \cdot A]$ by the weight vector $w$
and one to $k[y_1^{1/l}, \dots,y_n^{1/l}] $ so that $\phi$ preserves degrees.
The grading determines a partial order $>_{w}$
 on the set of monomials of $k[y_1^{1/l}, \dots,y_n^{1/l}]  $.

Put $I_A := \ker \phi$, which is  an ideal generated by 
finitely many binomials 
\[
g_1 =y^{c_1}-y^{d_1}, \dots, g_s=y^{c_s}-y^{d_s} \in k[y_1^{1/l}, \dots,y_n^{1/l}] .
\] 
Consider the ideal
\[
J:= I_A + \langle y_1 -1, \dots, y_n-1\rangle.
\]
This has a basis 
\begin{equation}\label{basis}
\{ g_1, \dots,g_s, g_{s+1} :=y_1 -1, \dots, g_{s+n}:=y_n-1\}.
\end{equation}

Now choose a total monomial order $>$ on $k[y_1^{1/l}, \dots,y_n^{1/l}] $
which refines $>_w$.
We apply  Buchberger's algorithm to the basis \eqref{basis} to find a Gr\"obner basis of $J$ with respect to $>$. For the basis \eqref{basis}, for every $1 \le i \le s+n$, we have
 $\phi(g_i) \in \Lambda$. 
This property of a basis is preserved in each step of the algorithm:
Taking an $S$-polynomial and each step of the division algorithm.
In particular, the output has also this property. 
Let  $\{ h_1, \dots, h_t\}$ be the output.  Then $\{ h_1, \dots, h_t\}$ is a Gr\"obner basis also
for $<_w$ (see for instance \cite[Corollary 1.9]{MR1363949}). Hence
$\{ \phi(h_1),\dots,\phi(h_t)\} \subset \Lambda$ is a Gr\"obner basis of $\fa_{l}$
with respect to $w$. 
\end{proof}

\begin{lem}\label{lem-monomial-ideal}
For a $d$-dimensional cone $\sigma \in \Delta_{A,l}$,
$\fb_\sigma = \initial_{w} \fa_{l}$ is a monomial ideal. Namely it is generated by some monomials in $k[(1/l) \cdot A]$.
\end{lem}

\begin{proof}
Let $J$ be as in the proof of Proposition \ref{prop-grobner}.
Then there exists a finite universal Gr\"obner basis $\{f_1,\dots,f_s\}$ of $J$, that is, a 
Gr\"obner basis for every total monomial order (see \cite[Corollary 1.3]{MR1363949}).
Then $\{\phi(f_1), \dots,\phi(f_s)\}$ is a Gr\"obner basis
of $\fa_{l}$ with respect to every $w \in \relint A_{\RR}^{\vee}$. 
 Since $\sigma$ is $d$-dimensional,
there exists $w \in \relint \sigma$
such that for every $i$, $\initial _w \phi(f_i)$ is a monomial.
It follows that $\fb_\sigma = \initial _w \fa_{l}$
is generated by the monomials $\initial _w \phi(f_i)$.
\end{proof}

\begin{lem}\label{lem-good-Grobner}
Let $\sigma \in \Delta_{A,l}$ be a $d$-dimensional cone,
$B_\sigma \subset (1/l) \cdot A$ the subset corresponding to the monomial ideal
$\fb_\sigma$, $w \in \relint \sigma $ and
\[
 \Lambda _ \sigma := \{ x^a -x^b \in \Lambda \mid x^a -x^b \ne 0, \, b \notin B_\sigma\}.
\]
Then there exists a finite subset  of $\Lambda_\sigma$ which is a Gr\"obner basis of
$\fa_{l}$ with respect to $w$. 
Hence $\Lambda_\sigma$ is also a Gr\"obner basis of
$\fa_{l}$ with respect to $w$.
\end{lem}

\begin{proof}
Let $>$ be a total monomial order on $k[(1/l) \cdot A]$ which refines $>_{w}$.
There exists a unique reduced Gr\"obner basis of $\fa_{l}$ for $>$ \cite[Lemma 11.13]{MR1363949},
which is contained in $\Lambda_{\sigma}$ by definition.
\end{proof}

\subsubsection{Coordinate rings of affine charts of $\FB_{(l)}(X)$}

Suppose that $A$ contains no nontrivial group, equivalently, $A_{\RR}^{\vee}$ is $d$-dimensional.
Then the affine open subvarieties $ U_{\sigma}$
associated to the $d$-dimensional cones $\sigma \in \Delta_{A,l}$ 
form an open covering of $\FB_{(l)}(X)$. The following proposition gives an explicit
expression of the coordinate rings of $U_{\sigma}$.

\begin{prop}\label{prop-coord-ring}
Let $\sigma \in \Delta _{A,l}$ be a $d$-dimensional cone,
$w \in \relint \sigma $ and $I \subset \Lambda_\sigma$
a Gr\"obner basis with respect to $w$.
Then the coordinate ring  of $U_\sigma$ is
\[
 k[U_\sigma]=k[x^{a-b} \mid x^a-x^b \in I]= k[x^{a-b} \mid x^a-x^b \in \Lambda_\sigma].
\]
\end{prop}

\begin{proof}
We first claim that for $x^a -x^b \in \Lambda_{\sigma}$,  $ w(a) > w (b) $. 
Indeed, on the contrary, if $w (a) \le w (b)$, then 
$\initial _w (x^{a} -x^{b})$, which is an element of $\fb_{\sigma}$,
 is either $ x^{a} -x^{b} $ or $x^{b}$.
Since $\fb_\sigma$ is a monomial ideal, we have $x^{b} \in \fb_\sigma$, which is a contradiction.

Hence the set
\[
\{ a-b \mid x^a-x^b \in I \} 
\]
 is contained in the half space $\{ w >0 \} \subset M_\RR$. 
 Put $R:= k[x^{a-b} \mid x^a-x^b \in I]$.
It follows that  
\[
\fm:= \langle x^{a-b} \mid x^a-x^b \in I \rangle \subset R
\]
 is  a maximal ideal.
Let $S:=\Spec R$, $s \in S$ the closed point defined by $\fm$
and $\cZ \subset  X_{(l)} \times S $ the subscheme defined by the ideal
\[
\fz:= \langle  x^{a}\otimes 1  -  x^{b} \otimes x^{a-b}  \mid  x^a-x^b \in I  \rangle 
\subset  k[(1/l) \cdot A] \otimes R .
\]
By construction, the fiber of $\cZ \to S$ over  $s$
has the defining ideal 
\[
\fz \cdot  k[(1/l) \cdot A] = \fz \cdot  ( k[(1/l) \cdot A] \otimes R/\fm ) =\langle  x^{a} \mid  x^a-x^b \in I \rangle  = \initial _w \fa_{l} = \fb_\sigma.
\]
Thus the fiber is in fact $Z_\sigma$.
Then since $ x^{-b}\otimes 1 $ is invertible in $  k[(1/l)\cdot M]\otimes k[M]$,
\[
 \fz \cdot ( k[(1/l)\cdot M] \otimes k[M]  )= \langle  x^{a-b}\otimes 1 -   1\otimes x^{a-b} \mid x^{a} -x^{b} \in I \rangle .
\]
Combining these, we see that $\cZ$ is a flat family which
generically coincides with the graph of $F^{(l)}_{T}$.

From the universality of the Hilbert scheme,
we have the birational morphism $\psi:S \to U_\sigma$ corresponding to 
the family $\cZ$. 
We will show that $\psi$ is an isomorphism. To do this, we need only to show that the map 
of the tangent spaces
\[
 T\psi:T_s S \to  T_ {Z_\sigma} U_\sigma \subset T_{Z_{\sigma}} \Hilb_{l^{d}} (X_{(l)}) =\Hom (\fb_\sigma, k[(1/l)\cdot A]/\fb_\sigma) 
\]
is injective.
Let  $(x^{a_v-b_v})_{v \in V}$ be a minimal set of generators of $R$
with $x^{a_{v}}-x^{b_{v}} \in I$.
Then the immersions
\[
 \epsilon_v: S_v := \Spec R / \langle x^{2a_v-2b_v}, x^{a_u-b_u} \mid  u \in  V \setminus \{v\} \rangle
\hookrightarrow S, \, v \in V
\]
are tangent vectors which are a basis of $T_s S$.
Let $\cZ_v \subset  X_{(l)}\times S_v $ be the restriction of the family $\cZ$ to $S_v$. It is defined by the ideal
\[
\langle  x^{a}\otimes 1  -   x^{b}\otimes x^{a-b} \mid  x^a-x^b \in I \rangle  
 \subset  k[(1/l) \cdot A]\otimes (R / \langle x^{2a_v-2b_v}, x^{a_u-b_u} \mid u \in V \setminus \{v\} \rangle).
\]
Hence for $v,u \in  V$,
\[
 (T\psi (\epsilon_v)) ( x^{a_u}) =  
\begin{cases}
x^{b_v}  & (v=u)  \\
0 & (v \ne u) .
\end{cases}
\]
Note that $x^{b_v}$ are  nonzero elements in $k[A]/\fb_\sigma$,
and hence $T\psi (\epsilon_v)$ are nonzero maps.
 Moreover from the above explicit expression, 
$ T\psi (\epsilon_v)$, $v \in V$, are  linearly independent.
Hence $T\psi$ is injective and so  $\psi$ is an isomorphism, which completes the proof.
\end{proof} 

A similar assertion for the $G$-Hilbert scheme was proved by
Craw, Maclagan and Thomas \cite[Theorem 5.2]{MR2356842}.

\subsection{Stability of the F-blowup sequence for a normal toric variety}\label{subsec-stability}

In this subsection, we suppose that $X$ is normal.
Fix an identification $M=\ZZ^d$ and give the standard Euclidean metric to $M_\RR = \RR^d$. 
Let $a_1,\dots,a_m \in A$ be the Hilbert basis, that is, the unique minimal set of generators. Then the ideal 
\[
\fm:= \langle x^{a_1}, \dots, x^{a_m} \rangle \subset k[A]
\]
 is the maximal ideal defining the distinguished point of $ X$. 
Set
\begin{gather*} 
\Theta := \bigcup_{i=1}^m ( A_\RR + a_i ) \subset A_\RR .
\end{gather*}
The ideal 
\[
 \fm _l : = \fm \cdot ((1/l) \cdot A)=  \langle x^{a_{1}}, \dots,x^{a_{m}}\rangle _{k[(1/l) \cdot A]} = k [\Theta \cap (1/l)\cdot M]
\]
defines the fiber of $F^{(l)}_{X} :X_{(l)} \to X$ over the distinguished point.
For every $d$-dimensional cone $\sigma \in \Delta_{A,l}$,
\[
 \fm _l \subset \fb_\sigma \text{ and } \Theta \cap (1/l) \cdot M \subset B_\sigma.
\]
Let  $D \in \RRpos$ be the diameter of $A_\RR \setminus \Theta$, that is,
\[
D:= \sup \{ |a-b| \mid a,b \in  A_\RR \setminus \Theta \}.
\]
Denote by $M_{\le D}$  the finite set of the lattice points $a \in M$ with $|a| \le D$. 
Choose a weight vector $w \in A_\RR^\vee \cap N $ such that 
for every $a \in M_{\le D} \setminus \{0\}$, $w (a) \ne 0$. 
Set $M^{w+}_{\le D} := \{a \in M_{\le D}|w(a) >0 \}$ and 
\[
\Xi:= \Theta  \cup \{ a \in A_\RR | \exists b \in A_\RR, \, a-b \in M^{w+}_{\le D}   \}.
\]
The following figures illustrate things introduced above in some two-dimensional case:

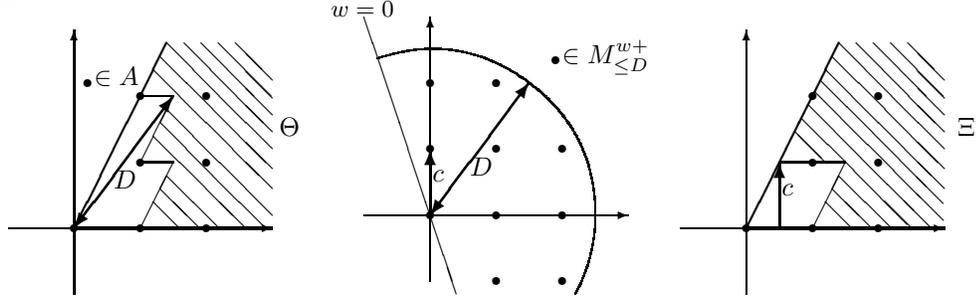
\setlength\unitlength{0.5pt}
\begin{tabular}{lccr}
\begin{minipage}{0.33\hsize}

\begin{center}
\begin{picture}(240,220)
\put(0,50){\vector(1,0){200}}
\put(50,0){\vector(0,1){200}}

\put(50,50){\circle*{5}}
\put(100,50){\circle*{5}}
\put(150,50){\circle*{5}}
\put(100,100){\circle*{5}}
\put(100,150){\circle*{5}}
\put(150,150){\circle*{5}}
\put(150,100){\circle*{5}}

\put(205,120){$\Theta$}

\put(60,160){\circle*{5}}
\put(65,158){$\in A$}

\put(80,80){$D$}

\thicklines

\put(50,50){\line(1,0){150}}
\put(50,50){\line(1,2){70}}

\put(50,50){\vector(3,4){75}}
\put(125,150){\vector(-3,-4){75}}

\thinlines

\put(100,50){\line(1,2){25}}
\put(100,100){\line(1,2){25}}
\put(100,100){\line(1,0){25}}
\put(100,150){\line(1,0){25}}

\put(105,60){\line(1,-1){10}}
\put(110,70){\line(1,-1){20}}
\put(115,80){\line(1,-1){30}}
\put(120,90){\line(1,-1){40}}

\put(105,110){\line(1,-1){10}}
\put(110,120){\line(1,-1){70}}
\put(115,130){\line(1,-1){80}}
\put(120,140){\line(1,-1){80}}

\put(105,160){\line(1,-1){10}}
\put(110,170){\line(1,-1){90}}
\put(115,180){\line(1,-1){85}}
\put(120,190){\line(1,-1){80}}
\put(133,190){\line(1,-1){67}}
\put(146,190){\line(1,-1){54}}
\put(159,190){\line(1,-1){41}}
\put(172,190){\line(1,-1){28}}
\put(185,190){\line(1,-1){15}}

\end{picture}
\end{center}

\end{minipage}

\begin{minipage}{0.33\hsize}

\begin{center}
\begin{picture}(260,220)(-25,-10)
\put(0,50){\vector(1,0){200}}
\put(50,0){\vector(0,1){200}}

\put(50,50){\circle*{5}}
\put(100,50){\circle*{5}}
\put(150,50){\circle*{5}}
\put(100,100){\circle*{5}}
\put(100,150){\circle*{5}}

\put(150,100){\circle*{5}}
\put(50,100){\circle*{5}}
\put(50,150){\circle*{5}}
\put(100,0){\circle*{5}}
\put(150,0){\circle*{5}}

\put(80,80){$D$}

\put(0,200){\line(1,-3){70}}

\put(-25,200){{\small $w=0$}}

\qbezier(10.4742,168.585)(72.6047,189.296)(125,150)
\qbezier(175,50)(175,112.5)(125,150)
\qbezier(175,50)(175,18.4193)(160,-9.3717)

\put(150,165){$\in M_{\le D}^{w+}$}
\put(145,167){\circle*{5}}

\put(52,75){$c$}

\thicklines
\put(50,50){\vector(0,1){50}}

\put(50,50){\vector(3,4){75}}
\put(125,150){\vector(-3,-4){75}}

\end{picture}
\end{center}
\end{minipage}

\begin{minipage}{0.33\hsize}
\begin{center}
\begin{picture}(240,220)(-20,0)
\put(0,50){\vector(1,0){200}}
\put(50,0){\vector(0,1){200}}

\put(50,50){\circle*{5}}
\put(100,50){\circle*{5}}
\put(150,50){\circle*{5}}
\put(100,100){\circle*{5}}
\put(100,150){\circle*{5}}
\put(150,150){\circle*{5}}
\put(150,100){\circle*{5}}

\put(210,120){$\Xi$}

\put(77,75){$c$}

\thicklines

\put(50,50){\line(1,0){150}}
\put(50,50){\line(1,2){70}}

\put(75,50){\vector(0,1){50}}

\thinlines

\put(100,50){\line(1,2){25}}
\put(75,100){\line(1,0){50}}

\put(105,60){\line(1,-1){10}}
\put(110,70){\line(1,-1){20}}
\put(115,80){\line(1,-1){30}}
\put(120,90){\line(1,-1){40}}

\put(91.667,133.333){\line(1,-1){82.7}}
\put(86.667,123.333){\line(1,-1){22.7}}
\put(81.667,113.333){\line(1,-1){12.7}}
\put(96.667,143.333){\line(1,-1){92.7}}
\put(101.667,153.333){\line(1,-1){97.7}}

\put(106,162){\line(1,-1){93}}
\put(110,170){\line(1,-1){90}}
\put(115,180){\line(1,-1){85}}
\put(120,190){\line(1,-1){80}}
\put(133,190){\line(1,-1){67}}
\put(146,190){\line(1,-1){54}}
\put(159,190){\line(1,-1){41}}
\put(172,190){\line(1,-1){28}}
\put(185,190){\line(1,-1){15}}

\end{picture}
\end{center}
\end{minipage}
\end{tabular}

\begin{lem}\label{lem-Xi}
Let $>$ be a total order on $M$ refining the partial order determined by $w$.
Then 
\[
\Xi = \{ a \in A_{\RR} \mid \exists b \in A_{\RR},\, a-b \in M,\, a>b \}.
\]
\end{lem}

\begin{proof}
Denote by $\Xi'$ the right hand side.
Obviously $ \Xi \subset \Xi' $.
Let $a \in \Xi'$ and $b \in A_{\RR}$ be such that $a-b \in M$ and $a > b$.
We will show that $a \in \Xi$.
If necessary, replacing $b$, we may suppose that $b \notin \Theta$. 
If $|a-b|>D$, then $a \in \Theta \subset \Xi$. 
Otherwise,  by definition, $a \in \Xi$. Thus in any case,
$a \in \Xi$. We have proved the lemma. 
\end{proof}

\begin{lem}\label{lem-stab}
Suppose  that $w $ is in $\relint \sigma$ for a $d$-dimensional cone $\sigma \in \Delta_{A,l}$.
Then $ B_\sigma =\Xi \cap (1/l)\cdot M$. 
\end{lem}

\begin{proof}
Let $>$ be a total order on $(1/l)\cdot A$ refining the partial order
determined by $w$. Then 
\begin{align*}
B_{\sigma} & = \{ a \in (1/l)\cdot A \mid \exists b \in (1/l)\cdot A ,\, a-b \in M,\, a>b  \}=\Xi \cap (1/l)\cdot M . 
\end{align*}
\end{proof}

We define  $L_{w} \subset M $ to be  the submonoid generated by
\[
 L'_{w}:=\{ c \in M^{w+}_{\le D} \mid  \exists a \in A_\RR, \exists b \in  A_\RR \setminus \Xi, \,  a-b =c \} .
\]
Since $a_{i} \in L'_{w}$, we have $A \subset L_{w}$.

\begin{prop}\label{prop-stab-coord}
Suppose that $w \in \relint \sigma$ for a $d$-dimensional $\sigma \in \Delta_{A,l}$. 
Then $k[U_\sigma] \subset k[L_{w}]$. Moreover if $l$ is sufficiently large,
then $k[U_\sigma]  = k[L_{w}]$.
\end{prop}

\begin{proof}
Let $\{a_1, \dots,a_m\} $ be the Hilbert basis of $A$ as above. 
Then there exists a subset 
\[
I :=  \{ x^{a_1} -1,\dots, x^{a_m}-1, x^{a_{m+1}} - x^{b_{m+1}} , x^{a_{m+2}} - x^{b_{m+2}} , \dots \}   \subset \Lambda_\sigma
\]
which is a Gr\"obner basis with respect to $w$. 
If necessary, replacing $b_{i}$, we may suppose that $b _{i} \notin \Theta$ for every $i$. 
Indeed if $b_{i} \notin \Theta$, then $b_i = b_i' + a_j $ for some $1 \le j \le m$
and $b_i' \in (1/l) \cdot A$.
Then  we can replace $x^{a_i}-x^{b_i} $ in $I$ with
\[
( x^{a_i}-x^{b_i}) + x^{b_i'}(x^{a_j} - 1) = x^{a_i} -  x^{b_i'}.
\]
Repeating this replacement iteratively, we obtain such a Gr\"obner basis $I$
with $b _{i} \notin \Theta$.

If for some $i > m$, $|a_i-b_i| >D$, then 
\[
x^ {a_i} \in \fm_l = \langle  x^{a_1},\dots, x^{a_m} \rangle_{k[(1/l)\cdot A]} . 
\]
Since it does not contribute to generation of $\fb_{\sigma}$,
 we may remove all such  binomials from $I$. 
From Proposition \ref{prop-coord-ring},
\[
 k[U_\sigma] =k[ x^{a-b} \mid x^a -x^b \in I] \subset k[L_{w}].
\] 

Let $c \in L_{w}'$, $a\in A_\RR$ and $b \in A_\RR \setminus \Xi$  be such that $a-b =c$. If  $l \gg 0$, then there exists $\delta  \in A_\RR$ such that
$b + \delta \in (1/l) \cdot A \setminus \Xi$. Then $x^{a+\delta} -x^{b+\delta} \in \Lambda_\sigma$ and
\[
 x^{c} =x^{(a+\delta) -(b+\delta)} \in k[U_\sigma] . 
\]
Thus $k[L_{w}] \subset k[U_{\sigma}]$, which completes  the proof.
\end{proof}

\begin{thm}\label{thm-stability}
Let $X$ be a normal toric variety. Then there exists $l_0 \in \ZZnonneg$ such that for every $l \in \ZZpos$, there exists a natural birational morphism $\FB_{(l_0)} (X) \to \FB_{(l)}( X)$, which is an isomorphism if $l \ge l_0$.
\end{thm}

\begin{proof}
Gluing the $\Spec k[L_{w}]$, we obtain a toric variety $Y$. 
The maps $ k[U_\sigma] \hookrightarrow k[L_{w}] $
induces a birational morphism $Y \to \FB_{(l)} (X) $ for every $l$, which is
an isomorphism for sufficiently large $l$. 
\end{proof}

\subsection{Boundedness of the F-blowup sequence for a non-normal toric variety}

\begin{thm}\label{thm-simultaneous}
Let $X$ be a (not necessarily normal) toric variety.
Then there exists a toric proper birational morphism 
$Y \to X$ such that for each 
$l \in \ZZpos$, it factors as $Y \to \FB_{(l)} (X) \to X$.
\end{thm}

\begin{proof}
The proof of the theorem is similar to that of Proposition \ref{prop-stab-coord}.
Again we may  suppose that $X = \Spec k[A]$ and
$A$ contains no nontrivial group. Let $a_{1},\dots,a_{m} \in A$ be the Hilbert
basis and let $M_{A,l} \subset (1/l)\cdot A$ be the set corresponding to $\fm_l=\fm\cdot ((1/l) \cdot A)$.
Explicitly, we have 
\begin{align*}
M_{A,l} 
 = \bigcup _{i} ( (1/l)\cdot A + a_{i})
 = \{ \sum _{i} n_{i} a_{i} \mid n_{i}  \in (1/l) \cdot \ZZnonneg \text{ and } \exists i, \, n_{i} \ge 1 \}.
\end{align*}
Then the set 
\begin{align*}
\Omega 
 := \bigcup_{l \ge 1}  (1/l)\cdot  A \setminus M_{A,l} 
\subset  \{ \sum _{i} n_{i} a_{i} \mid n_{i}  \in \QQ_{\ge 0}, \, n_{i} < 1 \}
\end{align*}
is bounded. Denote by $D$ the diameter of $\Omega$ and
define $M_{\le D}$ and $M^{w+}_{\le D}$ as in the preceding subsection.
Then for a generic $w \in A_{\RR}^{\vee}$, 
if $\sigma \in \Delta_{A,l}$ is a $d$-dimensional cone with $w \in \relint \sigma$,
by a similar argument as the proof of Proposition \ref{prop-stab-coord}, we can show
\[
 k[U_{\sigma}] \subset k[M_{\le D}^{w+}].
\]
We now claim that the affine toric varieties $\Spec k[M_{\le D}^{w+}]$
are glued together and become a toric variety which is proper and 
birational over $X$. Once it is shown, the theorem immediately follows.

For a generic $w \in A_{\RR}^{\vee}$, the set $C_{w} := 
\{w' \in A_{\RR}^{\vee} \mid \forall a \in M_{\le D}^{w+} , \, w'(a) \ge 0 \}$ 
is a $d$-dimensional rational polyhedral cone. 
Conversely for an interior point $w' $ of $C_{w}$, 
we have $M_{\le D}^{w+} =\{ a \in M_{\le D} \mid w' (a)>0   \} $. 
As $w$ varies, the cones $C_{w}$
form a fan $\Delta$ with $|\Delta|=A_{\RR}^{\vee}$. 
If $C_{w_{1}}$ and $C_{w_{2}}$ share a facet with the supporting hyperplane $ \{w \mid w(a) = 0 \} $
for a primitive element $a \in M_{\le D}$, then the ring
$k[ M_{\le D}^{w_{1}+} \cup \{\pm a \} ] = k[ M_{\le D}^{w_{2}+} \cup \{\pm a \} ] $
is localizations both of $k[ M_{\le D}^{w_{1}+} ] $ and of $ k[ M_{\le D}^{w_{2}+} ] $.
Hence we can glue $\Spec k[ M_{\le D}^{w_{1}+} ] $ and $\Spec k[ M_{\le D}^{w_{2}+} ]$
along $ \Spec k[ M_{\le D}^{w_{i}+} \cup \{\pm a \} ] $.
Thus we can glue all the $ \Spec k[ M_{\le D}^{w+}]$ together in a compatible way
and obtain a toric variety $Y$. 
This proves the above claim
and finishes the proof.
\end{proof}

\subsection{A two-dimensional normal toric singularity}

Now suppose that $X=\Spec k[A]$ is normal and two-dimensional.
 The fan $\minDelta$ associated to the minimal resolution
of $X$  is  described as follows. 
Let $H \subset N_\RR = \RR^2$ be the convex hull of 
$( A_\RR^\vee \cap N ) \setminus \{0\}$.
Then the one-dimensional cones of $\minDelta$ is
the half-lines through the lattice points on the boundary of $H$.
(See \cite[\S 2.6]{MR1234037}).

Choose a two-dimensional cone $\tau \in \minDelta$ and fix an identification $N = \ZZ^2$
so that $\tau$ is spanned by $(1,0)$ and $(0,1) $.
Then the cone $A_\RR^\vee$ is spanned by 
$ (-s,t) $ and $(u,-v)$ for some $s,t,u,v \in \ZZnonneg$ such that
$\gcd(s,t)=\gcd(u,v)=1$. Moreover we may and shall assume that $s < t$ and $u > v$.

Suppose that $M=\ZZ^2$ is the dual of  $N=\ZZ^2$ in the standard way:
The bilinear form $M \times N \to \ZZ$ is given by
$\langle (a,b),(c,d) \rangle =ac+bd$.
Then $A_\RR$ is spanned by $(v,u)$ and $(t,s)$.
In particular, $A_\RR$ is contained in the first quadrant $\RRnonneg^2$ and contains $(1,1)$.

For each $l \in \ZZpos$, we fix a weight vector $w=w_{l} \in \relint A_\RR^\vee \cap N$ as follows:
Put 
$w:= (n,n+1)$ for $n \gg 0$ so that  $w \in \relint \sigma$ for some
 two-dimensional cone $\sigma$ of  $\Delta_{A,l}$. 

\begin{prop}\label{prop-main}
For $l \gg 0$,  $k[U_\sigma]=k[x_1,x_2]$. In particular, we have $\sigma = \tau$.
\end{prop}

\begin{proof}
First, if  $(t,s)=(1,0)$ and $(v,u)=(0,1)$, 
then $X$ is an affine plane and  the proposition clearly holds.

Next, suppose that either $s>0$ or $v>0$, say $s>0$.
Let $c_1,c_{2} \subset M_\RR$ be the 1-dimensional 
cones spanned by $(t,s)$ and $(v,u)$ respectively.
Put
\begin{gather*}
Q_1 := \left (  \frac{ut}{ut-vs} ,  \frac{us}{ut-vs} \right) \in  c_1 \\
Q_2 := \left ( \frac{vt}{ut-vs} , \frac{ut}{ut-vs} \right) \in  c_2 \\
R_1 := Q_1  - (1,0) = \left(\frac{vs}{ut-vs}, \frac{us}{ut-vs} \right) \in c_2 \\
R_2 := Q_2 - (0,1) =  \left ( \frac{vt}{ut-vs} , \frac{vs}{ut-vs} \right) \in c_1. 
\end{gather*}


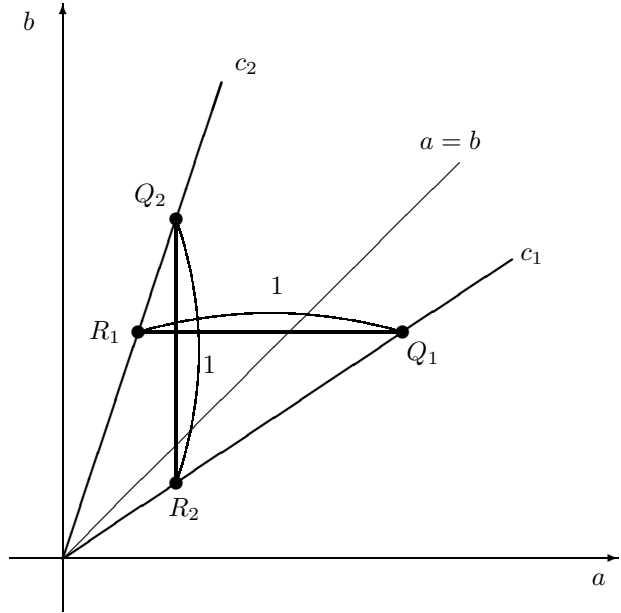
\setlength\unitlength{1pt}

\begin{center}
\begin{picture}(230,230)
\put(0,20){\vector(1,0){230}}
\put(20,0){\vector(0,1){230}}
\put(20,20){\line(1,1){150}}
\put(220,10){$a$}
\put(5,220){$b$}
\put(155,175){$a=b$}
\put(193,133){$c_{1}$}
\put(85,205){$c_{2}$}
\put(47,155){$Q_{2}$}
\put(60,36){$R_{2}$}
\put(150,95){$Q_1$}
\put(30,103){$R_1$}

\put(62.8,148.5){\circle*{5}}
\put(148.5,105.7){\circle*{5}}
\put(62.8,48.5){\circle*{5}}
\put(48.5,105.7){\circle*{5}}
\qbezier(148.5,105.7)(98.5,120)(48.5,105.7)
\put(98.5,120){1}
\qbezier(62.8,148.5)(80,98.5)(62.8,48.5)
\put(73,90){1}

\thicklines
\put(20,20){\line(1,3){60}}
\put(20,20){\line(3,2){170}}
\put(62.8,148.5){\line(0,-1){100}}
\put(148.5,105.7){\line(-1,0){100}}

\end{picture}
\end{center}

Let $\Xi$ be as in \S \ref{subsec-stability}.
By definition, $Q_i \in \Xi$, $i=1,2$.


\begin{claim}\label{claim1}
There is no point  $S  \in A_\RR$ except $R_1$
such that  $ w(Q_{1}) \ge w (S)$ and $Q_{1}-S \in M$. Similarly there is no point  $S  \in A_\RR$ except $R_2$  
such that  $ w(Q_{2}) \ge w (S)$ and $Q_{2}-S \in M$.
\end{claim}

\begin{proof}[Proof of the claim]
We give  only a proof of the first assertion.
A proof of the second assertion is similar.
Let $P$ be the intersection of 
$c_2$ and the line through $Q_1$ with the slope $\frac{-n}{n+1}$,
and let $O$ be the origin.

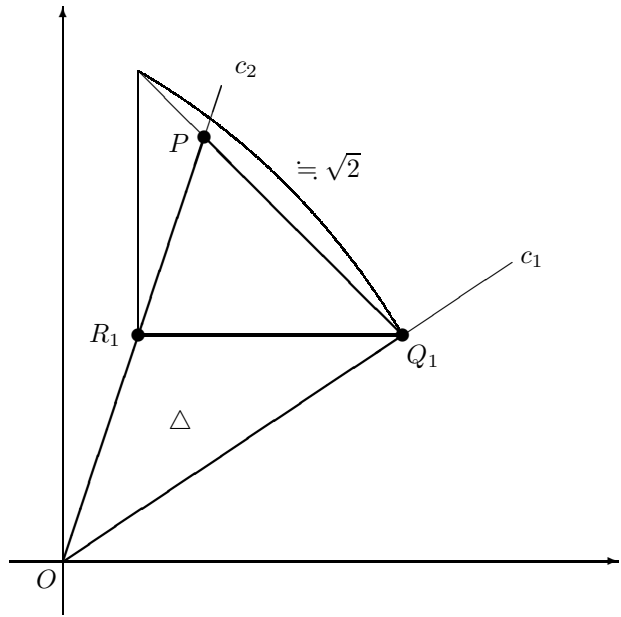
\begin{center}
\begin{picture}(230,230)
\put(0,20){\vector(1,0){230}}
\put(20,0){\vector(0,1){230}}

\put(193,133){$c_{1}$}
\put(85,205){$c_{2}$}

\put(150,95){$Q_1$}
\put(30,103){$R_1$}

\put(148.5,105.7){\circle*{5}}
\put(48.5,105.7){\circle*{5}}
\put(73.5,180.7){\circle*{5}}

\qbezier(48.5,205.7)(110,170)(148.5,105.7)
\put(108.5,165.7){$ \fallingdotseq  \sqrt 2 $}

\put(20,20){\line(1,3){60}}
\put(20,20){\line(3,2){170}}

\put(148.5,105.7){\line(-1,1){100}}
\put(48.5,105.7){\line(0,1){100}}

\put(60,70){$\triangle$}
\put(10,10){$O$}
\put(60,175){$P$}

\thicklines
\put(148.5,105.7){\line(-1,1){75}}
\put(20,20){\line(1,3){53.5}}
\put(20,20){\line(3,2){128.5}}
\put(148.5,105.7){\line(-1,0){100}}

\end{picture}
\end{center}

If  such  $S$ existed, since $w(Q_1) \ge w(S)$, $S$ would be in the triangle $OPQ_1$.
Since $Q_1-S$ is none of $(\pm 1, 0)$, $(0,\pm 1)$ and $(\pm 1,\pm 1)$, we have
 $|Q_1 -S| \ge 2 $. 
 But since $|Q_1 -R_1|< 2 $
and $|P- Q_1|< 2 $,
$S$ is not in the triangle $PR_1Q_1$.
(Concerning the second assertion of the claim,  if we define $P' \in c_1$ to be the intersection of $c_1$ and the line through $Q_2$ with the slope $\frac{-n}{n+1}$, then 
since $s >0$, $|P'- Q_2| < \sqrt 2 $. But $Q_{2}-S$ is none of $(\pm 1, 0)$ and  $(0,\pm 1)$. Thus $S$ is not in the triangle $P'R_2Q_2$.)

Then we will show that $S$ is not in the triangle $\triangle:= OR_1 Q_1$ either.
First translate $\triangle$ so that $Q_1$ maps to the origin.
Then rotate clockwise it $90^\circ $ around the origin
and get a new triangle, denoted $\bar \triangle$.
This fits into the fan $\minDelta$ as follows:

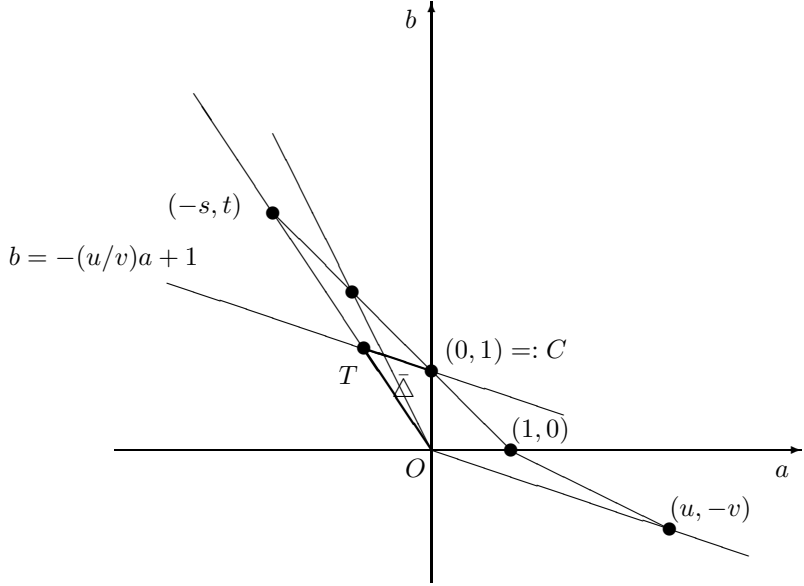
\begin{center}
\begin{picture}(260,220)
\put(0,50){\vector(1,0){260}}
\put(120,0){\vector(0,1){220}}
\put(250,40){$a$}
\put(110,210){$b$}

\put(150,50){\circle*{5}}
\put(210,20){\circle*{5}}
\put(210,25){$(u,-v)$}
\put(150,50){\line(2,-1){60}}
\put(120,50){\line(3,-1){120}}
\put(120,50){\line(-1,2){60}}
\put(120,80){\circle*{5}}
\put(90,110){\circle*{5}}
\put(60,140){\circle*{5}}
\put(20,140){$(-s,t)$}
\put(150,50){\line(-1,1){90}}
\put(120,50){\line(-2,3){90}}

\put(120,80){\line(-3,1){100}}
\put(120,80){\line(3,-1){50}}
\put(-40,120){$b= -(u/v)a +1$}

\put(125,85){$(0,1)=:C$}
\put(150,55){$(1,0)$}

\put(105,70){$\bar \triangle$}

\put(94.3,88.6){\circle*{5}}
\put(85,75){$T$}
\put(110,40){$O$}

\thicklines

\put(120,50){\line(0,1){30}}
\put(120,50){\line(-2,3){27}}
\put(120,80){\line(-3,1){27}}
\end{picture}
\end{center}

If we denote by $T$ the intersection of  the lines
$\{ b = - (t/s)a\}$ and $\{b= -(v/u)a +1 \}$
and put $C := (0,1)$,
then $\bar \triangle$ is the triangle $OTC$. By the description of $\minDelta$, 
there is no lattice point in $\bar \triangle$ except $O$ and $C$.
Therefore $S$ is not in $\triangle$.
This completes the proof of the claim.
\end{proof}

We next claim that $R_2 \notin \Xi$.
If this was not the case, there would exist $V \in A_\RR$ such that $w (R_2) \ge w (V) $ and 
$R_2 -V \in M$.  Since $Q_{1}-R_{2} \in A_{\RR}$, $V+Q_1 -R_2 \in A_\RR$. 
Moreover we have
\begin{gather*}
w ( Q_1)\ge w(V+Q_1 -R_2 ),\\
V+Q_{1}-R_{2} \notin c_{2}, \ (\text{in particular, } V+Q_{1}-R_{2} \ne R_{1}) \text{ and} \\
Q_1-(V+Q_1 -R_2 )= R_2-V \in M.
\end{gather*}
From Claim  \ref{claim1}, the point $V+Q_1 -R_2$ does not exist,
a contradiction.

Let $L'_{w}$ and $ L_{w}$ be as in \S \ref{subsec-stability}.
We have $(0,1) = Q_2 -R_2 \in L _{w}$. Similarly $(1,0) \in L _{w}$.
From Proposition \ref{prop-stab-coord}, 
 $k[x_1,x_2] \subset k[U_\sigma] =k[L_{w}]$. 
 Now to show  $k[U_\sigma] = k[x_1,x_2]$, it suffices to show that for every
 $c,d\in \ZZpos$, we have $(c,-d),(-d,c) \notin L'_{w}$.
On the contrary, suppose that either $(c,-d) \in L'_{w}$ or $(-d,c) \in L'_{w}$, say $W:=(c,-d) \in L'_{w}$.
Then by definition, there exists $E \in A_{\RR}$ with $E-W \in A_\RR \setminus \Xi$.
It is easy to see that there exists $G \in A_\RR$ such that $E-G \in c_1$ and $E-W-G\in c_2$. As in the following figure, $|E-G-W| > |Q_2| $.

{\unitlength=0.2mm
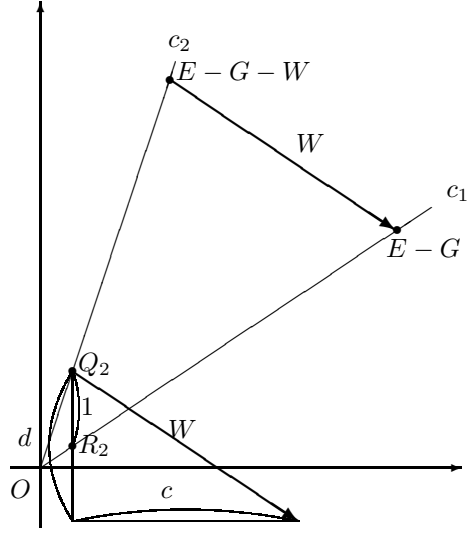
\begin{center}
\begin{picture}(300,350)
\put(0,40){\vector(1,0){300}}
\put(20,0){\vector(0,1){350}}

\put(20,40){\line(1,3){90}}
\put(20,40){\line(3,2){260}}

\put(290,220){$c_{1}$}
\put(105,320){$c_{2}$}

\put(41.4,104.3){\circle*{5}} \put(45,104.3){$Q_2$} 
\put(41.4,104.3){\line(0,-1){100}} 
\put(41.4,4.3){\line(1,0){150}} 
\put(41.4,54.3){\circle*{5}}\put(45,50){$R_2$}

\qbezier(41.4,104.3)(50,79.3)(41.4,54.3)
\put(47,75.3){$1$}

\qbezier(41.4,104.3)(10,54.3)(41.4,4.3)
\put(5,54.3){$d$}

\qbezier(41.4,4.3)(116.4,20)(191.4,4.3)
\put(100,20){$c$}

\put(190,250){$W$}
\put(105,60){$W$}
\put(0,20){$O$}

\put(250,180){$E-G$}
\put(257,198){\circle*{5}}
\put(106,298){\circle*{5}} \put(110,298){$E-G-W$} 

\thicklines
\put(41.4,104.3){\vector(3,-2){150}} 

\put(106,298){\vector(3,-2){150}} 

\end{picture}
\end{center}}
Therefore  $E-W-G \in \Xi$ 
and hence $E-W \in \Xi$. This is a contradiction.
We have completed the proof.
\end{proof}

The following is a direct consequence of the above proposition.

\begin{thm}\label{thm-minimal}
Let $X$ be a normal two-dimensional toric variety.
Then for $l \gg 0$, $\FB_{(l)} (X)$ is the minimal resolution of $X$.
\end{thm}

\begin{rem}
 In Section \ref{sec-G}, comparing the $G$-Hilbert scheme and the F-blowup,
 we will prove a similar result, which is however
 valid only for tame quotient singularities.
\end{rem}

\begin{rem}\label{rem-minimal}
A slight change in the arguments gives more. Suppose that
 $X$ is a two-dimensional (not necessarily normal) toric variety and has 
 an isolated singularity, and that the normalization of $X$ is NOT smooth. 
 Then for $l \gg 0$, $\FB_{(l)} (X)$ is the minimal resolution.
 Compare this with Proposition \ref{prop-isolated}.
\end{rem}

\subsection{An isolated singularity whose normalization is smooth}

\begin{prop}\label{prop-isolated}
Suppose that $X= \Spec k[A]$ has an  isolated singularity 
and the normalization $\tilde X$ of $X$ is smooth, so $\tilde X \cong \AA^d$.
 Then for $l \gg 0$,
$\FB_{(l)}( X)$ is the blowup of $\tilde X$ at the origin.
\end{prop}

\begin{proof}
We may suppose that $A \subset \ZZnonneg^d $ 
and that $A^c := \ZZnonneg ^d\setminus A$ is finite.
Let 
$e_i := (0, \dots , \stackrel{i}{\breve{1}},\dots,0 )$ be the $i$-th generator of $\ZZnonneg^d$.
For $l \gg 0 $, we have
$(1/l) \cdot A^{c} \subset [0,1]^{d} \subset \RR^{d}$.
Then $e_i \in (1/l) \cdot A$ and $x_i -1 \in \fa_{l}$.
Therefore for any weight vector $w \in \relint A_\RR^\vee \cap N = \ZZpos^d$, 
we have $x_i \in \initial _w \fa_{l} $.
Let $\fd \subset \initial _w \fa_{l}$ be the ideal of  $k[(1/l)\cdot A]$ generated by $x_i$, $i=1,\dots,d$.
The corresponding subset of $(1/l)\cdot A$ is
\[
(1/l)\cdot A \setminus \left ( \{ 0/l,1/l,\dots,(l-1)/l\}^d \cup \bigcup _{i=1} ^d(1/l)\cdot ( A^c)+  e_i \right) .
\]
Suppose that for some $i_0$,  $w(e_{i_0}) < w( e_i)$, $1 \le i \le d$, $i \ne i_0$.
Then for $b \in A^c $ and $1 \le i \le d$, $i \ne i_0$, we have $x^{e_i +b/l}-x^{e_{i_0} +b/l} \in \fa_{l}$.
So $x^{e_i + b/l} \in \initial _w \fa_{l}$.
Hence if we define a semigroup $H$ by 
\[
H:= (1/l) \cdot A \setminus  (\{0/l,1/l,\dots,(l-1)/l\}^d \cup  (1/l)\cdot (A^c) + e_{i_0}) ,
\]
then we have the inclusion of ideals $k[H] \subset \initial _w \fa_{l}$.
But since $k[(1/l)\cdot A]/k[H]$ has length $l^d$, we have $\initial _w \fa_{l} = k[H]$. 
Now it is easy to see that $w \in \relint \sigma$ for some $d$-dimensional cone $\sigma \in \Delta_{A,l}$ and  
$k[U_\sigma] = k[ x ^{e_{i_0}} , x^{e_i -e_{i_0}} \mid 1\le i \le d, \, i\ne i_0 ]$.
This proves the proposition.
\end{proof}

\subsection{The Whitney umbrella}\label{subsec-Whitney}

Put 
\[
\ZZnonneg^{2} \supseteq A := \{ (0,n) \mid n \in 2 \ZZnonneg \}  \cup 
\{ (l,m) \mid l \in \ZZpos , \ m \in \ZZnonneg \}.
\] 
Then $A$ is a submonoid. The toric singularity
$X := \Spec k[A]$ is known as the Whitney umbrella.
It is defined by the equation $x^2z-y^2=0$ as a subvariety of $\AA^3_k$.
It is not an isolated singularity, but the normalization $\tilde X$ of $X$
is an affine plane.

\begin{prop}
\begin{enumerate}
\item For even $l \ge 2$,  $\FB_{(l)} (X) \cong \AA^2_k$.
\item For odd $l \ge 3$, $\FB_{(l)} (X) $ 
is the blowup of $\AA^2_k$ at the origin.
\end{enumerate}
\end{prop}

\begin{proof}
\begin{enumerate}
\item We have $x_1-1,x_2 -1  \in \fa_{l}$.
From this, $x_1,x_2 \in \initial_w \fa$ for any weight vector
$w \in \ZZpos^2$. Since the ideal $\langle x_1  ,x_2\rangle \subset k[(1/l) \cdot A]$ has colength $l^2$,
we have $\initial _w \fa  = \langle x_1  ,x_2 \rangle$.
Therefore the fan $\Delta_{A,l}$ consists of $A_\RR^\vee$
and all its faces.
From Proposition \ref{prop-coord-ring}, 
 $\FB_{(l)} ( X) = \Spec k[x_1,x_2 ] = \AA^2 _k$.

\item  We first observe
\[
x_1-1, x_2^{2}-1, x_1^{1/l}x_2-x_1^{1/l}, x_1^{1/l}x_2^{(l+1)/l}-x_1^{1/l}x_2^{1/l},   x_1 x_2^{a/l} -x_2^{(l+a)/l}  \in \fa_{l},
\]
 where $a$ runs over the positive odd numbers. 
If $x _1 >_w x_2$,
then 
\[
\initial _w \fa = \langle x_1, x_2^{2},  x_1^{1/l}x_2, x_1^{1/l}x_2^{(l+1)/l},  x_1 x_2^{a/l} \mid a \text{ runs over positive odd numbers}  \rangle.
\]
Hence for $\sigma \in \Delta _{A,l}$ with $w \in \relint \sigma$, we have
 $k[U_\sigma] = k[x_2 ,x_1x_2^{-1}]$.
If $x_1 <_wx_2$, then 
\[
\initial _w \fa = \langle x_1, x_2^{2},  x_1^{1/l}x_2, x_1^{1/l}x_2^{(l+1)/l},  x_2^{1+a/l} \mid a \text{ runs over positive odd numbers}  \rangle,
\]
and $k[U_\sigma] = k[x_1 ,x_1^{-1}x_2]$.
We have proved the assertion.
\end{enumerate}
\end{proof}

Thus for a non-normal toric variety $X$, the sequence  $\FB_{(1)} (X)$, $\FB_{(2)} (X)$, \dots, does not generally stabilize. 
However at least for the Whitney umbrella, 
if $k$ is a perfect field of  characteristic $p>0$, then the sequence  $\FB_{1} (X)$, $\FB_{2} (X)$, \dots, stabilizes (to different toric varieties depending on whether $p$ is even or not.).

\section{$G$-Hilbert scheme vs.\ F-blowup}\label{sec-G}

In this section,  we suppose that $k$
is an algebraically closed field of characteristic $p >0$.

\subsection{Global quotients}

Let $M $ be a smooth variety of dimension $d$ and $G$ a finite group with $p \nmid   |G|$.
Suppose that $G$ acts on $M$ effectively. The $G$-Hilbert scheme, denoted
$\GHilb (M)$, is the closure of the set of free orbits in $\Hilb_{|G|} (M)$.

\begin{lem}
There exists a natural immersion 
\begin{align*} 
\GHilb (M) & \hookrightarrow \Hilb_{|G| \cdot q^{d}} (M_{e}) \\
Z  &\mapsto (F^{e}_{M})^{-1}(Z) .
\end{align*}
\end{lem}

\begin{proof}
For simplicity, we give the proof only for the case where $M$ is affine, say $M=\Spec R$.
The proof for the general case is similar.
Let $\cZ \subset  M\times \GHilb(M) $ be the universal family of $G$-clusters
and $\cZ_{e} \subset   M_{e}\times \GHilb(M)$ its pull-back
by the morphism $  M_{e}\times\GHilb(M) \to  M\times\GHilb(M) $.
Since the last morphism is flat, $\cZ_{e}$ is a flat family over $\GHilb(M)$.
We have the corresponding morphism 
$\iota:\GHilb (M)  \to \Hilb_{|G| \cdot q^{d}} (M_{e}) $, which is clearly injective.
Consider a nonzero tangent vector
\[
 \epsilon : \Spec k[t]/\langle t^{2} \rangle \to \GHilb (M),
\]
say at a $G$-cluster $Z \in \GHilb (M)$. Let $I_{Z} \subset R$ be the defining
ideal of $Z$. Then $\epsilon $ corresponds to a map
\[
 \epsilon' :I_{Z} \to R/I_{Z}.
\]
Let $\cZ_{\epsilon} \subset  M\times \Spec k[t]/\langle t^{2} \rangle $
be the pull-back of $\cZ$ by $\epsilon$ and 
$J \subset R[t]/\langle t^{2} \rangle$ its defining ideal.
If $f+g t \in J$, $f\in I_{Z} $, $g \in R$, then 
$\epsilon'(f) $ is $g$ modulo $I_{Z}$. Since $\epsilon'$ is a nonzero map,
there exist $f_{0}\in I_{Z}$ and $g_{0} \in R \setminus I_{Z}$
with $f_{0} +g_{0} t \in J$.

Let $\cZ_{e,\epsilon} \subset  M_{e} \times\Spec k[t]/\langle t^{2} \rangle $ be the pull-back of $\cZ_{e}$ by $\epsilon$. This corresponds to the image  $\bar \epsilon$ of the tangent vector $\epsilon$ on $\Hilb_{|G| \cdot q^{d}} (M_{e}) $.
The defining ideal of $\cZ_{e,\epsilon}$  is $J \cdot  R^{1/q}[t]/\langle t^{2} \rangle$. 
The $\bar \epsilon$ corresponds to a map 
\[
I_{Z}R^{1/q} \to R^{1/q} / I_{Z}R^{1/q} .
\]
Since $f_{0} +g_{0} t \in J \cdot  R^{1/q}[t]/\langle t^{2} \rangle$,
the map sends $f_{0}$ to $ g_{0} $.  Since $R^{1/q}$ is flat over $R$,
$I_{Z}R^{1/q} \cap R  = I_{Z}$. Hence $g_{0 } \notin I_{Z}R^{1/q} $.
It follows that $\bar \epsilon$ is a nonzero tangent vector. 
This shows that $\iota $ is an immersion. 
\end{proof}

\begin{prop}\label{prop-GtoF}
Put $X:=M/G$.
For every $e$, the natural morphism $\GHilb (M) \to X$ factors
as 
\[
 \GHilb ( M) \to \FB_{e} (X) \to X
\]
and the map 
\[ 
\GHilb (M) \to \FB_{e} (X)
\]
sends a $G$-cluster $Z \subset M$ to the cluster $(F ^{e})^{-1}(Z) /G \subset X$.
(In general, the quotient variety $X$ is not a scheme but only an algebraic space.
But it will not cause any problem.)
\end{prop}

\begin{proof}
We first note that $(F ^{e})^{-1}(Z) /G $ has the expected length $q^{d}$.
Indeed $(F ^{e})^{-1}(Z) $ has length $q^{d} \cdot |G|$ and $H^{0}(\cO_{(F ^{e})^{-1}(Z) })$
is, as a $G$-representation, isomorphic to the direct sum of $q^{d}$
copies of the regular representation.
Hence the length of $(F ^{e})^{-1}(Z) /G $ is $ \dim _{k} H^{0}(\cO_{Z^{[q]}})^{G} =
q^{d}$.

Let $\cZ _{e} \subset  M_{e}\times\GHilb (M) $ be as above.
Then $\cZ_{e}/G \subset  X_{e}\times\GHilb (M) $
is a flat family which generically parameterizes $(F^{e})^{-1}(x) $, $x \in X_{\sm}$.
Now from the universality, we obtain the desired morphism 
$\GHilb (M) \to \FB_{e} (X)$.
\end{proof}

\begin{thm}\label{thm-GtoF-isom}
For sufficiently large $e$,  the morphism
$\GHilb (M) \to \FB_{e} (X)$ is an isomorphism. 
If $G$ is abelian, then $q > | G|$ is sufficient for the assertion to hold.
\end{thm} 

\begin{proof}
Let $\phi$ denote the morphism $\GHilb ( M) \to \FB_{e} (X)$.
We first show that $\phi$ is injective.
Take two distinct $G$-clusters $Z ,Z' \in \GHilb (M) $.
If $Z_{\red} \ne Z'_{\red} $, then $\phi (Z) \ne \phi( Z')$.
So we may suppose that $Z_{\red}= Z'_{\red}$.
Choose a point $m \in Z_{\red} =Z'_{\red}$ and connected components
$Y \subset Z$ and $Y' \subset Z'$ with $Y_{\red}=Y'_{\red}=m$.
Let $\hat M = \Spec k[[ x_{1}, \dots,x_{d} ]]$ be the completion of $M$ at $m$.
We denote closed subschemes of $\hat M$ corresponding to $Y$ and $Y'$ by the same letters. 
Let 
$I_{Y}, \, I_{Y'} \subset k[[ x_{1}, \dots,x_{d} ]]$  be the defining ideals of $Y, \, Y'$ respectively,
and $H \subset G$ the stabilizer of $m$.
Without loss of generality, we may and shall suppose that $H$ acts on $k[[x_{1}, \dots,x_{d}]]$
linearly.
Since $I_{Y} / I_{Y} \cap I_{Y'}$ is a nonzero $H$-representation,
there exists an irreducible subrepresentation $V \subset I_{Y}$ with $V \cap I_{Y'} = \{ 0\}$.
Put 
\[
P_{e} := \bigoplus_{n_{i} < 1}k \cdot x_{1}^{n_{1}} \cdots  x_{d}^{n_{d}} \subset k[[x_{1}^{1/q}, \dots, x_{d}^{1/q}]].
\] 
From Lemma \ref{lem-Bryant}, there exists an irreducible representation
 $U \subset P_{e}$ which is dual to $V$. 
(Note that $P_{e}$ may not be an $H$-subrepresentation.)
Then $V \cdot U  \subset  I_{Y} \cdot k[[x_{1}^{1/q}, \dots, x_{d}^{1/q}]]$ contains 
a trivial irreducible subrepresentation $V'$.    From Lemma \ref{last-lem},
$V'$ is not contained in $I_{Y'} \cdot k[[x_{1}^{1/q},\dots, x_{d}^{1/q}]]$.
Hence
\[
(I_{Y} \cdot k[[x_{1}^{1/q},\dots, x_{d}^{1/q}]])^{H} \ne (I_{Y'} \cdot k[[x_{1}^{1/q},\dots, x_{d}^{1/q}]])^{H}.
\] 
It follows that $\phi$ is injective.

Now it suffices to show that every nonzero tangent vector on $\GHilb (M)$
maps to a nonzero one on $\FB_{e} (X)$.
Take a nonzero tangent vector $ \epsilon \in T_{Z} \GHilb (M)$. 
With the notation as above, $\epsilon$  is identified with a nonzero $H$-equivariant 
$k[[x_{1}, \dots,x_{d}]]$-linear map
\[
 \epsilon: I_{Y} \to k[[x_{1}, \dots,x_{d}]] / I_{Y} .
\]
We extend this map to 
\[
\tilde \epsilon: I_{Y} \cdot k[[x_{1}^{1/q},\dots, x_{d}^{1/q}]] \to
k[[x_{1}^{1/q},\dots, x_{d}^{1/q}]]/I_{Y} \cdot k[[x_{1}^{1/q},\dots, x_{d}^{1/q}]] .
\]
The image of $\epsilon$ on $\FB_{e} (X)$ corresponds to the restriction of $\tilde \epsilon$,
\[
\bar \epsilon: (I_{Y} \cdot k[[x_{1}^{1/q},\dots, x_{d}^{1/q}]])^{H} \to
(k[[x_{1}^{1/q},\dots, x_{d}^{1/q}]]/I_{Y} \cdot k[[x_{1}^{1/q},\dots, x_{d}^{1/q}]])^{H}.
\]
What is left is to show that the last map is nonzero.

Since $\epsilon$ is nonzero, there exists an irreducible representation $V\subset I_{Y}$
such that $\epsilon|_V$ is injective.
Again from Lemma \ref{lem-Bryant}, there exists an irreducible subrepresentation
$U \subset P_{e}$ which is dual to $V$. 
Let $V' \subset U \cdot  V  \subset k[[x_{1}^{1/q},\dots, x_{d}^{1/q}]]$ be
a trivial subrepresentation. 
From Lemma \ref{last-lem},  the image $\bar \epsilon  (V')$ of $V'$, 
which is a trivial subrepresentation of $ \tilde \epsilon(U \cdot V) = U \cdot \epsilon (V)$,  is nonzero.
 We have completed the proof.
\end{proof}

\begin{lem}\label{lem-Bryant}
\begin{enumerate}
\item\label{Bryant-1}
With the above notation, for sufficiently large $e$,  $P_{e}$
contains all irreducible $H$-representations.
\item In addition if $H$ is abelian and if for each $i$, 
$ k \cdot x_{i}$ is stable under the $H$-action, then the preceding assertion
holds for  $q > | H |$.
\end{enumerate}
\end{lem}

\begin{proof}
\begin{enumerate}
\item
It follows from Bryant's theorem \cite{MR1206130} that
for sufficiently large $l$, 
the subset $Q_{l} \subset k[[x_{1},\dots,x_{d}]]$ of the polynomials of degree $<l$ 
contains the regular representation. 
If $Q_{l,e} \subset k[[ x_{1}^{1/q},\dots,x_{d}^{1/q}  ]]$ denotes the 
subset of the polynomials of degree $ < l/q$, then 
$Q_{l,e}$ and $Q_{l}$ are isomorphic as $H$-sets.
Now we easily see that $Q_{l,e}$ also contains the regular representation.
If $q \ge l$, then  $Q_{l,e} \subset P_{e}$ and the assertion follows.

\item 
If $D$ denotes the set of all irreducible representations,
the map  $\alpha  : D \to D$, $V \mapsto V^{\otimes q}$ has an inverse $\alpha^{-1}$. 
In particular $\alpha$ and $\alpha^{-1}$ are bijective. 
An irreducible representation in $P_{e}$
is the image of one in 
\[
P'_{e} := \bigoplus_{n_{i} < q}k \cdot x_{1}^{n_{1}} \cdots  x_{d}^{n_{d}}\subset k[[x_{1},\dots,x_{d}]] .
\]
Hence it suffices to show that $P'_{e}$ 
 contains every irreducible representation. 
Again from Bryant's theorem, every irreducible representation 
is of the form $ \bigotimes k \cdot x_{i}^{\otimes n_{i}}$ for some $n =(n_{1},\dots,n_{d})$,
$0 \le  n_{i} < | H |$. Hence $P'_{e} $ contains
every irreducible representations. We have completed the proof.
\end{enumerate}
\end{proof}

\begin{lem}\label{last-lem}
Let $I \subset k[[x_{1},\dots,x_{d}]]$ be an ideal with $\sqrt I = \langle x_{1} , \dots, x_{d} \rangle $
and let $P_{e} \subset k[[x_{1}^{1/q} , \dots , x_{d}^{1/q}]]$ be as above. 
Then for  $ f \in k[[x_{1},\dots,x_{d}]] \setminus I$ and 
$g \in P_{e} $, we have 
$f g \not \in I\cdot k[[x_{1}^{1/q},\dots,x_{d}^{1/q}]]$.
\end{lem}

\begin{proof}
From the assumption $\sqrt I = \langle x_{1} , \dots, x_{d} \rangle $, 
we may replace $k[[x_{1},\dots,x_{d}]]$ and
 $k[[x_{1}^{1/q},\dots,x_{d}^{1/q}]]$
with $k[x_{1},\dots,x_{d}]$ and   $k[x_{1}^{1/q},\dots,x_{d}^{1/q}]$ respectively. 
Then, in the case where $f$ and $g$ are monomials and $I$
is a monomial ideal, the assertion is obvious.
Therefore we will reduce the problem to the monomial case.
Fix a total monomial order $>$ on $k[x_{1}^{1/q},\dots,x_{d}^{1/q}]$. This induces a total monomial order
on $k[x_{1},\dots,x_{d}]$, denoted also by $>$. 
 Let $ x^{m_{i}}$, $i \in 1,\dots, n$, be the standard monomials of $I$
  (that is, those monomials not in $\initial_{>} I$), 
whose images form a vector space basis of $k[x_{1},\dots,x_{d}]/I$ (for instance ,see \cite[page 1]{MR1363949}). 
Then modulo $I$, $f$ is equal to some nonzero polynomial $\tilde f := \sum_{i=1}^{n} a_{i}x^{m_{i}}$, $a _{i } \in k$.
To obtain a contradiction, suppose $ \tilde f g \in I\cdot k[x_{1}^{1/q},\dots,x_{d}^{1/q}]$.
Then 
\[
\initial_{>} (\tilde f g ) = \initial_{>} (\tilde f  )\cdot  \initial_{>} ( g ) \in \initial_{>} ( I\cdot k[x_{1}^{1/q},\dots,x_{d}^{1/q}]) 
= (\initial_{>} I)\cdot k[x_{1}^{1/q},\dots,x_{d}^{1/q}].
\]
From the assertion for the monomial case, this does not hold. The lemma follows.
\end{proof}

The above results give nontrivial consequences both on
the $G$-Hilbert scheme and on the F-blowup.

\begin{cor}
\begin{enumerate}
\item The F-blowup sequence of $M/G$
\[
\FB_{0} (M/G ), \, \FB_{1} (M/G) , \, \FB_{2} (M /G), \dots
\]
stabilizes and leads to a modification of $M/G$ which is isomorphic to $\GHilb (M)$. 
\item The iterate Frobenius morphisms $F^{e}_{M/G}$, $e \in \ZZpos$ of $M/G$
are simultaneously flattened by the modification $\GHilb (M) \to M/G $.
\end{enumerate}
\end{cor}

\begin{proof}
Direct consequences of Proposition \ref{prop-GtoF} and Theorem \ref{thm-GtoF-isom}.
\end{proof}

\begin{cor}
The $G$-Hilbert scheme $\GHilb (M)$ depends
only on the quotient variety $M/G$. 
Precisely, if $(G',M')$ satisfies the same assumption as $(G,M)$ does and if 
$M'/G' \cong M/G$, then $\GHilb (M) \cong \Hilb^{G'}(M')$.
\end{cor}

\begin{proof}
Since $\FB_{e} (M/G)$ depends only on $M/G$ and $e$,
the corollary follows 
from Theorem \ref{thm-GtoF-isom}.
\end{proof}

\begin{cor}
With the notation as above, suppose that 
$M$ is 2-dimensional.
Then $\GHilb (M)$ is the minimal resolution of $M/G$.
\end{cor}

\begin{proof}
The corollary was proved in \cite{MR1916656,MR1420598,MR1815001}
in the case where for every $g \in G$, the fixed point locus $M^{g}$ has
codimension $\ge 2$. 
(Although they worked in characteristic zero, Ishii's arguments \cite{MR1916656} hold valid
in positive characteristic under our tameness condition.)
Since the assumption that $\codim M^{g} \ge 2$
does not give any restriction on the singularities of $M/G$
and since $\GHilb(M)$ depends only on $M/G$,  
$\GHilb(M)$ is the minimal resolution also in the general case. We have proved the corollary.
(If $G$ is abelian, then $M/G$ has 2-dimensional normal toric singularities and
the corollary follows also from Proposition \ref{prop-etale} and 
Theorems \ref{thm-minimal} and \ref{thm-GtoF-isom}.)
\end{proof}

\begin{cor}
Let $G \subset SL(3,k)$ be a finite subgroup with $p \nmid |G|$.
Then for sufficiently large $e$, $\FB_{e} (\AA^{3}_{k}/G)$ is a crepant resolution.
\end{cor}

\begin{proof}
This follows from the fact that $\GHilb(\AA^{3}_{k})$ is a crepant resolution \cite{MR1824990,MR1838978}.
\end{proof}

\subsection{Deligne-Mumford stacks}

Let $\cX$ be a separated and smooth Deligne-Mumford stack over $k$.
Suppose that $\cX$ is tame, that is, every $k$-point of $\cX$
has an automorphism group of order prime to $p$.
Then $\cX$ is locally expressed as a quotient stack 
$[M/G]$. Following Abramovich's observation, Chen and Tseng \cite{MR2407221} verified
that locally defined $G$-Hilbert schemes, $\GHilb (M)$, can be 
``glued'' together and the result is isomorphic to an irreducible component of $Q(\cO_\cX/\cX/X)$.
Here $X$ is the coarse moduli space and $Q(\cO_\cX/\cX/X)$ is Olsson-Starr's Quot algebraic
space associated to the structure sheaf $\cO_{\cX}$ (see \cite{MR2007396}).
Denote it by $\Hilb ' (\cX)$.  Then there exists a  natural proper birational morphism
$\Hilb' (\cX) \to X$. Now it is straightforward to restate
results in the preceding subsection in this generalized situation.

\begin{cor}\label{cor-stack}
For each $e \in \ZZnonneg$, the  morphism $\Hilb' ( \cX) \to X$
factors as $\Hilb'(\cX) \to \FB_{e} (X) \to X$.
Moreover for $e \gg 0$, the above morphism $\Hilb' (\cX) \to \FB_{e} (X)$ is
an isomorphism. 
Hence $\Hilb' (\cX)$ depends only on the coarse moduli space $X$.
\end{cor}

%
%

\end{document}